\documentclass{amsart}

\usepackage[a4paper, total={5.4in, 8in}]{geometry}

\usepackage[utf8]{inputenc}
\usepackage[british]{babel}

\usepackage{amsmath, amsthm, amssymb}
\usepackage{mathtools, mathrsfs}
\newcommand{\mb}{\mathbb}
\newcommand{\mc}{\mathcal}
\newcommand{\mf}{\mathfrak}
\newcommand{\ms}{\mathscr}

\usepackage{array}
\newcolumntype{C}{>{\displaystyle}>{$}c<{$}}
\newcolumntype{L}{>{\displaystyle}>{$}l<{$}}
\newcommand\T{\rule{0pt}{2.6ex}}
\newcommand\B{\rule[-1.2ex]{0pt}{0pt}}
\newcommand\M{\T\B}

\usepackage{hyperref}
\usepackage{bm}

\usepackage{tikz-cd, pgfplots}
\usetikzlibrary{positioning}
\graphicspath{{figures/}{figures/stable/}{figures/mond/}}
\tikzcdset{
  cells={font=\everymath\expandafter{\the\everymath\displaystyle}}
}

\newenvironment{funcion*}%
	{\begin{equation*}\begin{tikzcd}[row sep=-1mm]}
	{\end{tikzcd}\end{equation*}}

\newenvironment{diagram}%
	{\begin{equation}\begin{tikzcd}}
	{\end{tikzcd}\end{equation}}

\newenvironment{diagram*}%
	{\begin{equation*}\begin{tikzcd}}
	{\end{tikzcd}\end{equation*}}

\newcommand{\la}{\left\langle}
\newcommand{\ra}{\right\rangle}

\newcommand{\p}{\partial}
\newcommand{\cvf}[1]{\frac{\p}{\p #1}}

\DeclareMathOperator{\codim}{codim}
\DeclareMathOperator{\Der}{Der}
\DeclareMathOperator{\GCD}{GCD}
\DeclareMathOperator{\Iso}{Iso}
\DeclareMathOperator{\Lift}{Lift}
\DeclareMathOperator{\mult}{mult}
\DeclareMathOperator{\ord}{ord}
\DeclareMathOperator{\rk}{rk}
\DeclareMathOperator{\Sp}{Sp}
\DeclareMathOperator{\supp}{supp}

\newtheorem{corollary}{Corollary}[section]
\newtheorem{conjecture}[corollary]{Conjecture}
\newtheorem{definition}[corollary]{Definition}
\newtheorem{lemma}[corollary]{Lemma}
\newtheorem{proposition}[corollary]{Proposition}
\newtheorem{theorem}[corollary]{Theorem}
\newtheorem{example}[corollary]{Example}
\newtheorem{remark}[corollary]{Remark}

\title{Deformations of corank $1$ frontals}
\author{C. Mu\~noz-Cabello, J.J. Nu\~no-Ballesteros, R. Oset Sinha}

\address{Departament de Matem\`{a}tiques,
Universitat de Val\`encia, Campus de Burjassot, 46100 Burjassot,
Spain}
\email{chmuca@alumni.uv.es}
\email{Raul.Oset@uv.es}

\address{Departament de Matemàtiques,
Universitat de Val\`encia, Campus de Burjassot, 46100 Burjassot,
SPAIN.
Departamento de Matemática, Universidade Federal da Paraíba CEP 58051-900, João Pessoa - PB,
BRAZIL}
\email{Juan.Nuno@uv.es}

\thanks{Work of Juan J. Nuño-Ballesteros and R. Oset Sinha partially supported by Grant PID2021-124577NB-I00 funded by MCIN/AEI/ 10.13039/501100011033 and by ``ERDF A way of making Europe"}

\begin{document}

\begin{abstract}
We develop a Thom-Mather theory of frontals analogous to Ishikawa's theory of deformations of Legendrian singularities but at the frontal level, avoiding the use of the contact setting. In particular, we define concepts like frontal stability, versality of frontal unfoldings or frontal codimension. We prove several characterizations of stability, including a frontal Mather-Gaffney criterion, and of versality. We then define the method of reduction with which we show how to construct frontal versal unfoldings of plane curves and show how to construct stable unfoldings of corank 1 frontals with isolated instability which are not necessarily versal. We prove a frontal version of Mond's conjecture in dimension 1. Finally, we classify stable frontal multigerms and give a complete classification of corank 1 stable frontals from $\mathbb C^3$ to $\mathbb C^4$. 
\end{abstract}

\maketitle

\section{Introduction}

The study of frontal mappings has flourished rapidly in the last decade.
Roughly speaking a frontal is a mapping $f\colon N\to Z$ where $N$ and $Z$ are $n$ and $(n+1)$-dimensional manifolds such that the image of $N$ has a well defined tangent hyperplane at each point.
More precisely, $f$ is a frontal if it admits a Legendrian lift $\tilde f\colon N\to PT^*Z$ such that $f=\pi\circ\tilde f$, where $\pi$ is the canonical fibration.
When the Legendrian lift is an immersion we say that $f$ is a wave front.
The concept of frontals was first introduced by Fujimori, Saji, Umehara and Yamada in \cite{Shoichi} (see also \cite{ZakalyukinKurbatskii}) and since then it has been of great interest to differential geometers, singularists and contact topologists.
The fact of having a well defined normal at each point allows one to study differential geometric properties and invariants in singular spaces (\cite{ChenPeiTakahashi, MartinsNuno, MurataUmehara, OsetSaji, SUYGaussBonnet}), on the other hand, when studying contact and symplectic topology front singularities are unavoidable \cite{Casals_Murphy} and understanding the generic (or stable) situations is crucial.

In \cite{Ishikawa05}, Ishikawa developed the analogue of the Thom-Mather theory for corank one Legendrian singularities and he stated the main notions like infinitesimal deformations, stability, versality, etc. Our purpose in this paper is to construct a Thom-Mather theory of singularities of frontals, but downstairs, at the level of frontals, and thus, avoiding the use of the contact setting.
In particular, we consider deformations that come from unfoldings $F$ of the frontal $f$.
We show that such unfoldings $F$ come from a deformation of its Legendrian lift $\tilde f$ if and only if $F$ is frontal as a mapping.
Taking local charts of $N$ and $Z$ we study map germs $f\colon (\mb{K}^n,S)\to(\mb{K}^{n+1},0)$ under $\ms{A}$-equivalence, i.e. smooth changes of coordinates in source and target. Here, smooth means $C^\infty$ when $\mb{K}=\mb{R}$ or holomorphic when $\mb{K}=\mb{C}$.
The case of frontal surfaces ($n=2$) was studied in a previous paper (\cite{FrontalSurfaces}) where analytic/toplogical invariants were defined and characteristations of finite frontal codimension were given, amongst other interesting results on surfaces, using some of the definitions and results that will be given in this paper.

In Section \ref{unfolding} we define the concept of frontal stability and versality.
We define a frontal codimension and prove that a frontal is stable if and only if it has frontal codimension 0.
We also give a characterisation of versality analogous to Mather's versality theorem.
Section \ref{mathergaffney} gives a geometric criterion for stability, a frontal Mather-Gaffney criterion which states that a frontal is stable if and only if it has isolated instability.
Sections \ref{frontal reductions} and \ref{stable unfoldings} are devoted to show how to construct stable frontals as frontal versal unfoldings of plane curves or as a well defined sum of frontal unfoldings.
We define the frontal reduction of an $\mathscr A_e$-versal unfolding of a plane curve and prove that it is, in fact, a versal frontal unfolding.
As a by product we relate the frontal codimension of a plane curve with its $\ms{A}_e$-codimension and prove the frontal Mond conjecture (stated in \cite{FrontalSurfaces}) in dimension 1, which says that the frontal codimension is less than or equal to the frontal Milnor number (the number of spheres in a stable deformation) with equality if the germ is quasi-homogeneous.
We also give a method to construct stable unfoldings which are not necessarily versal.
We then turn our attention to characterizing stability of frontal multigerms defining a frontal Kodaira-Spencer map which also yields a tangent space to the iso-singular locus (the manifold along which the frontal is trivial).
Finally we use our methods to obtain a complete list of stable 3-dimensional frontals in $\mb{C}^4$.
Note that generic wave fronts were classified by Arnol'd in \cite{Arnold_CWF} and, on the other hand, Ishikawa classified stable Legendrian maps (which may have different projected frontals), but, until now, a complete classification of stable frontals was only known for $n=1$ (\cite{Arnold_CWF}) and $n=2$ (\cite{Nuno_CuspsAndNodes}).

For technical reasons in order to use Ishikawa's results we restrict ourselves to the case of frontals whose Legendrian lift has corank 1.

\section{Frontal map-germs}
Let $W$ be a smooth manifold of dimension $2n+1$.
A field of hyperplanes $\Delta$ over $W$ is a contact structure for $W$ if, for all $w \in W$, there exist an open neighbourhood $U \subseteq W$ of $w$ and a $\sigma \in \Omega^1(U)$ such that
\begin{enumerate}
	\item $\rk \sigma_w=1$;
	\item the fibre $\Delta_w$ of $\Delta$ at $w$ is $\ker \sigma_w$;
	\item $(\sigma\wedge d\sigma \wedge \stackrel{(n)}{\dots} \wedge d\sigma)_w\neq 0$.
\end{enumerate}
We call $\sigma$ the local contact form of $W$, and define a \textbf{contact manifold} as a pair $(W,\Delta)$, where $\Delta$ is a contact structure on $W$.
Given a smooth manifold $Z$ of dimension $n+1$, a submersion $\pi\colon W \to Z$ is a \textbf{Legendrian fibration} for $(W,\Delta)$ if, for all $w \in W$,
	\[(d\pi_w)^{-1}(T_{\pi(w)}Z) \subseteq \ker \sigma_w.\]

\begin{example}\label{projective cotangent}
	Let $W=PT^*\mb{K}^{n+1}$ be the projectivised cotangent bundle of $\mb{K}^{n+1}$, and $(z,[\omega]) \in W$.
	The differential $1$-form
		\[\alpha=\omega_1\,dz^1+\dots+\omega_{n+1}\,dz^{n+1}\]
	defines a contact structure on $W$.
	The projection $W \to \mb{K}^{n+1}$ given by $(z,[\omega]) \mapsto z$ is a Legendrian fibration under this contact structure.
\end{example}

\begin{definition}
	Let $\pi\colon W \to Z$, $\pi'\colon W' \to Z'$ be Legendrian fibrations.
	A diffeomorphism $\Psi\colon W \to W'$ between contact manifolds is
	\begin{enumerate}
		\item a \textbf{contactomorphism}, if $\Delta'=d\Psi(\Delta)$;
		\item a \textbf{Legendrian diffeomorphism} if it is a contactomorphism and there exists a diffeomorphism $\psi\colon Z \to Z'$ such that $\psi\circ \pi=\pi'\circ \Psi$.
	\end{enumerate}
	We say $W$ is contactomorphic to $W'$ if there is a contactomorphism $\Psi\colon W \to W'$.
\end{definition}

A well-known result by Darboux states that any two contact manifolds $W,W'$ of the same dimension admit a local diffeomorphism $\Psi\colon W \to W'$ such that $\Delta'=d\Psi(\Delta)$ (see e.g. \cite{Arnold_I}, \S 20.1).
In particular, if $\dim W=2n+1$, $W$ is locally contactomorphic to the contact manifold described in Example \ref{projective cotangent}; therefore, we can restrict oruselves to the setting given in Example \ref{projective cotangent}.

Let $N \subseteq \mb{K}^{n+1}$ be an open subset.
A mapping $F\colon N \to PT^*\mb{K}^{n+1}$ is \textbf{integral} if $F^*\alpha=0$.

\begin{definition}
	A smooth mapping $f\colon N^n \to Z^{n+1}$ is \textbf{frontal} if there exist an integral mapping $F\colon N \to W$ and a Legendrian fibration $\pi\colon W \to Z$ such that $f=\pi \circ F$.
	If $F$ is an immersion, we say $f$ is a \textbf{wave front}.
	Similarly, a hypersurface $X \subset Z$ is \textbf{frontal} (resp. a \textbf{wave front}) if there exists a frontal map (resp. wave front) $f\colon N \to Z$ such that $X=f(N)$.
\end{definition}

\begin{definition}
	Let $S \subset N$ be a finite set.
	A smooth multigerm $f\colon (N,S) \to (Z,0)$ is \textbf{frontal} if it has a frontal representative $f\colon N \to Z$.
	Given a hypersurface $X \subset Z$, $(X,z)$ is a \textbf{frontal} hypersurface germ if there exists a frontal map germ $f\colon (N,S) \to (Z,z)$ such that $(X,z)=f(N,S)$.
\end{definition}

Let $F\colon N \to PT^*\mb{K}^{n+1}$ be an integral map and $f=\pi\circ F$: there exist $\nu_1,\dots,\nu_{n+1} \in \ms{O}_n$ such that
\begin{equation}\label{eq: interpretation of frontal}
	0=F^*\alpha=\sum^{n+1}_{i=1} \nu_1 d(Z_i\circ f)=\sum^{n+1}_{i=1}\sum^n_{j=1}\nu_i\frac{\p f_i}{\p x_j}\,dx^j,
\end{equation}
where $Z_1,\dots,Z_{n+1}$ are coordinates for $\mb{K}^{n+1}$.
Setting $\nu=\nu_1\,dZ_1+\dots+\nu_{n+1}\,dZ_{n+1}$, this is equivalent to $\nu(df\circ \xi)=0$ for all $\xi \in \theta_n$.
Since $PT^*\mb{K}^{n+1}$ is a fibre bundle, we can find for each pair $(z,[\omega]) \in PT^*\mb{K}^{n+1}$ an open neighbourhood $Z \subset \mb{K}^{n+1}$ of $z$ and an open $U \subseteq \mb{K}P^{n+1}$ such that $\pi^{-1}(Z)\cong Z\times U$.
Therefore, $F$ is contact equivalent to the mapping $\tilde f(x)=(f(x),[\nu_x])$, known as the \textbf{Nash lift} of $f$.

If we assume that $\Sigma(f)$ is nowhere dense in $N$, the differential form $\nu$ is uniquely determined by $f$, giving us a one-to-one correspondence between $f$ and $\tilde f$.
Such a frontal map is known as a \textbf{proper frontal} map (according to Ishikawa \cite{Ishikawa_Survey}).
We also define the \textbf{integral corank} of a proper frontal as the corank of its Nash lift.

For the rest of this article, we shall assume all frontal map germs are proper.
Note that the notion of topological properness (i.e. the preimage of a compact subset is compact) is not used throughout this article.

\begin{example}
	Let $f\colon (\mb{K}^n,0) \to (\mb{K}^{n+1},0)$ be the smooth map germ given by
	\begin{align*}
		f(x_1,\dots,x_n)=(x_1^2,\dots,x_n^2,2x_1^{p_1}+\dots+2x_n^{p_n}); && p_1,\dots,p_n > 1
	\end{align*}
	It is easy to see that $f$ has corank $n$ and the singular set $\Sigma(f)$ is nowhere dense in $\mb{K}^n$.
	Furthermore, the assumption that $p_1,\dots,p_n > 1$ implies that the Jacobian ideal of $f$ is generated by $x_1x_2\dots x_n$, and thus it is a proper frontal map germ by Proposition \ref{Ishikawa criterion frontality} below.
	In particular, the differential $1$-form
	\begin{align*}
  		\nu_{(x_1,\dots,x_n)}=p_1x_1^{p_1-2}\,dX^1+\dots+p_nx_n^{p_n-2}\,dX^n-dX^{n+1},
  	\end{align*}
	verifies that $\nu(df\circ\xi)=0$ for all $\xi \in \theta_n$, and has corank equal to the number of $p_i$ that are greater than $3$.
	Therefore, the integral corank of $f$ is also equal to the number of $p_i$ greater than $3$.
	In particular, $f$ is a wave front when all $p_i$ are equal to $3$.
\end{example}

\begin{proposition}[\cite{Ishikawa_Survey}, Lemma 2.3]\label{Ishikawa criterion frontality}
	Let $f\colon (\mb{K}^n,S) \to (\mb{K}^{n+1},0)$ be a map germ.
	If $f$ is frontal, then the Jacobian ideal $J_f$ of $f$ is principal (i.e. it is generated by a single element).
	Conversely, if $J_f$ is principal and $\Sigma(f)$ is nowhere dense in $(\mb{K}^n,S)$, then $f$ is a proper frontal map germ.
\end{proposition}

If $f$ has corank $1$, we may choose local coordinates in the source and target such that
	\begin{align}\label{eq: prenormal corank 1}
		f(x,y)=(x,p(x,y),q(x,y)); && x \in \mb{K}^{n-1},\, y \in \mb{K}
	\end{align}
in which case $J_f$ is the ideal generated by $p_y$ and $q_y$, and we recover the following criterion by Nuño-Ballesteros \cite{Nuno_CuspsAndNodes}:

\begin{corollary}\label{pq criterion frontality}
	Let $f\colon (\mb{K}^n,S) \to (\mb{K}^{n+1},0)$ be a frontal map germ of corank $1$, and choose coordinates in the source and target such that $f$ is given as in Equation \eqref{eq: prenormal corank 1}.
	Then $f$ is a frontal map germ if and only if either $p_y|q_y$ or $q_y|p_y$.
\end{corollary}

We shall say that $f$ is in \textbf{prenormal form} if it is given as in Equation \eqref{eq: prenormal corank 1} with $q_y=\mu p_y$ for some $\mu \in \ms{O}_n$, in which case the Nash lift becomes
\begin{equation}\label{prenormal nash lift}
	\tilde f=\left(f,\frac{\p q}{\p x_1}-\mu\frac{\p p}{\p x_1},\dots,\frac{\p q}{\p x_{n-1}}-\mu\frac{\p p}{\p x_{n-1}},\mu\right)
\end{equation}
In particular, note that if $\ord_y(q)=\ord_y(p)+1$, then $\ord_y(\mu)=1$, and $f$ is a wave front.

\section{Lowering Legendrian equivalence}\label{unfolding}
The first strides in the classification of frontal mappings were done by Arnol'd and his colleagues in a series of articles published in the 1970s and 1980s.
In his work, he established a notion of equivalence native to Legendrian maps (known as \emph{Legendrian equivalence}) and developed a classification of all simple, stable wave fronts (see \cite{Arnold_I}, Chapter 21).

Ishikawa extended Arnol'd's theory of Legendrian equivalence to the broader class of integral mappings in \cite{Ishikawa05}, defining a notion of infinitesimal stability and showing that an integral map of corank at most $1$ is Legendrian stable if and only if it is infinitesimally stable.
He also showed that all Legendrian stable integral mappings of corank at most $1$ belong to a special family called open Whitney umbrellas, giving a characterisation of stable umbrellas in terms of a certain $\mb{K}$-algebra $Q$.

The goal of this section is to formulate a notion of frontal stability and versality that does not require the use of contact geometry.

\begin{remark}\label{rank is stored in the X}
	Let $f\colon (\mb{K}^n,0) \to \mb{K}^{n+1}$ be a proper frontal map germ with Nash lift $\tilde f=f\times [\nu]$.
	Since $[\nu]$ is an equivalence class in a projective space, there exists a $1 \leq i \leq n+1$ such that $\nu_i$ is non-vanishing, so we can rewrite Equation \eqref{eq: interpretation of frontal} as
	\begin{equation}\label{concentrated rank}
		d(Z_i\circ f)=-\frac{\nu_1}{\nu_i}\,d(Z_1\circ f)-\dots-\widehat{d(Z_i\circ f)}-\dots-\frac{\nu_{n+1}}{\nu_i}\,d(Z_{n+1}\circ f),
	\end{equation}
	where the hat symbol denotes an ommited summand.
	We then define local coordinates $X,Y,P$ on $PT^*\mb{K}^{n+1}$ such that $f_i=Y\circ f$ and
	\begin{align*}
		f_j=		&X_j\circ f,	& P_j	&=\displaystyle\frac{\nu_j}{\nu_i}		& (j=1,\dots,i-1);	\\
		f_{j+1}=	&X_j\circ f,	& P_j	&=\displaystyle\frac{\nu_{j+1}}{\nu_i} 	& (j=i,\dots,n+1).
	\end{align*}
	These are known as the \textbf{Darboux coordinates} of $PT^*\mb{K}^{n+1}$.
	In particular, Equation \eqref{concentrated rank} implies that the mapping $X\circ f=(X_1\circ f, \dots, X_n \circ f)$ shares the same singular set with $f$.
	Therefore, there exists a representative $X\circ f\colon U \to V$ of $X\circ f$ which is immersive outside of a nowhere dense subset $K$ of $U$.
\end{remark}

\begin{definition}
	Let $S,S' \subset \mb{K}^n$ be finite sets.
	Two integral map germs
	\begin{align*}
		F\colon (\mb{K}^n,S) \to (PT^*\mb{K}^{n+1},w), && F'\colon (\mb{K}^n,S') \to (PT^*\mb{K}^{n+1},w')
	\end{align*}
	are \textbf{Legendre equivalent} if there exists a diffeomorphism $\phi\colon (\mb{K}^n,S) \to (\mb{K}^n,S')$ and a Legendrian diffeomorphism $\Psi\colon (PT^*\mb{K}^{n+1},w) \to (PT^*\mb{K}^{n+1},w')$ such that $F'=\Psi\circ F\circ \phi^{-1}$.
\end{definition}

Arnol'd showed in \cite{Arnold_I}, \S 20.4 that a Legendrian diffeomorphism $\Psi\colon W \to W'$ is locally determined by a choice of Legendrian fibrations in the source and target, and a diffeomorphism $\psi$ between the base spaces.
Nonetheless, his proof was based on the fact that a Legendrian diffeomorphism preserves the fibres, and no explicit expression is given for $\Psi$.

\begin{theorem}\label{diffeomorphism fixes contactomorphism}
	Given a diffeomorphism $\psi\colon Z \to Z'$, the mapping
	\begin{funcion*}
		T^*Z \arrow[r] 					& T^*Z'	\\
		(z,\omega) \arrow[r,maps to]	& (\psi(z),\omega\circ d\psi_{\psi(z)}^{-1})
	\end{funcion*}
	induces a Legendrian diffeomorphism $\Psi\colon (PT^*Z,\Delta) \to (PT^*Z',\Delta')$.
\end{theorem}

\begin{proof}
	Let $(z,\omega) \in T^*Z$: since $\psi$ is a diffeomorphism, $\omega \circ d\psi^{-1}_{\psi(z)} \neq 0$ and $\Psi$ is a well-defined diffeomorphism.
	Furthermore, it is clear that
	\begin{equation}
		\pi'\circ\Psi=\psi\circ\pi \label{legendrian equivalence}
	\end{equation}
	by construction.
	Therefore, we only need to show that $d\Psi_q(\Delta_q)=\Delta'_{\Psi(q)}$.

	Let $q=(z,[\omega])$ and $v \in \Delta_q$.
	Since $\pi$ is a submersion, $(\omega\circ d\pi_q)(v)=0$, and it follows from \eqref{legendrian equivalence} that
		\[(\omega\circ d\psi_{\psi(z)}^{-1}\circ d\pi'_{\Psi(q)})[d\Psi_q(v)]=0 \implies d\Psi_q(v) \in \Delta'_{\Psi(q)}\]
	Conversely, let $w \in \Delta'_{\Psi(q)}$.
	Since $\Psi$ is a diffeomorphism, there exists a unique $v \in T_qPT^*Z$ such that $w=d\Psi_q(v)$.
	By definition of $\Delta'$, we have
		\[(\omega\circ d\psi^{-1}_{\psi(z)}\circ d\pi'_{\Psi(q)})(w)=0\]
	By \eqref{legendrian equivalence}, this implies that $(\omega\circ d\pi_q)(v)=0$, from which follows that $w \in d\Psi_q(\Delta_q)$.
\end{proof}

\begin{remark}
	Let $\psi_t\colon (\mb{K}^{n+1},0) \to (\mb{K}^{n+1},0)$ be a smooth $1$-parameter family of diffeomorphisms.
	Given $t$ in an open neighbourhood $U \subseteq \mb{K}$ of $0$, we know by Theorem \ref{diffeomorphism fixes contactomorphism} that we can lift $\psi_t$ onto a Legendrian diffeomorphism $\Psi_t\colon (PT^*\mb{K}^{n+1},w) \to (PT^*\mb{K}^{n+1},0)$.
	Since $\pi\colon PT^*\mb{K}^{n+1} \to \mb{K}^{n+1}$ is a fibre bundle and $\mb{K}^{n+1}$ is a paracompact Hausdorff space, $\pi$ is a fibration (see \cite{Spanier}, Corollary 2.7.14), so it verifies the homotopy lifting property.
	Therefore, the $1$-parameter family $\Psi_t$ defined in this way is, indeed, a lift of the family $\psi_t$.
\end{remark}

\begin{corollary}
	\label{A-equivalence is compatible with frontals}
	Let $f,g\colon (\mb{K}^n,S) \to (\mb{K}^{n+1},0)$:
	\begin{enumerate}
		\item if $f$ is $\ms{A}$-equivalent to $g$ and $f$ is frontal, $g$ is frontal;
		\item if $f$ and $g$ are frontal, $\tilde f$ is Legendrian equivalent to $\tilde g$ if and only if $f$ is $\ms{A}$-equivalent to $g$.
	\end{enumerate}
\end{corollary}

\begin{proof}
	Assume that $f$ is frontal: there exist an integral map germ $F\colon (\mb{K}^n,S) \to PT^*\mb{K}^{n+1}$ such that $f=\pi\circ F$, where $\pi$ is the canonical bundle projection.
	Now let $\phi\colon (\mb{K}^n,S) \to (\mb{K}^n,S)$, $\psi\colon (\mb{K}^{n+1},0) \to (\mb{K}^{n+1},0)$ be diffeomorphisms such that $g=\psi\circ f\circ \phi^{-1}$: by Theorem \ref{diffeomorphism fixes contactomorphism}, we can lift $\psi$ onto a Legendrian diffeomorphism $\Psi\colon PT^*\mb{K}^{n+1} \to PT^*\mb{K}^{n+1}$.
	Therefore, the map $G=\Psi\circ F\circ \phi^{-1}$ is an integral map such that $\pi\circ G=g$, and $g$ is frontal.
	This proves the first item.

	For the second item, the ``only if" is proved in a similar fashion.
	For the ``if", let $\phi\colon (\mb{K}^n,S) \to (\mb{K}^n,S)$ and $\Psi\colon (PT^*\mb{K}^{n+1},w) \to (PT^*\mb{K}^{n+1},w)$ be diffeomorphisms such that $\tilde g=\Psi\circ \tilde f\circ \phi^{-1}$, with $\Psi$ Legendrian.
	By definition of Legendrian diffeomorphism, there exists a diffeomorphism $\psi\colon (\mb{K}^{n+1},0) \to (\mb{K}^{n+1},0)$ such that $\pi\circ \Psi=\psi\circ \pi$, from which follows that
		\[g=\pi\circ\tilde g=\pi\circ\Psi\circ \tilde f\circ \phi^{-1}=\psi\circ\pi\circ\tilde f\circ\phi^{-1}=\psi\circ f\circ \phi^{-1},\]
	proving the second item.
\end{proof}

\subsection{Unfolding frontal map germs}
The theory of Legendrian equivalence describes homotopic deformations of a pair $(\pi, F)$ via integral deformations, deformations $(F_u)$ of $F$ which are themselves integral for any fixed $u$.
Nonetheless, frontal deformations often fail to preserve the frontal nature across the parameter space, as showcased in Example \ref{E6 no frontal} below.

\begin{example}\label{E6 no frontal}
	Let $\gamma\colon (\mb{K},0) \to (\mb{K}^2,0)$ be the plane curve $t\mapsto (t^3,t^4)$.
	The $1$-parameter deformation $\gamma_s(t)=(t^3+st,t^4)$ verifies that $\gamma_s$ is frontal for all $s \in \mb{K}$.
	If $\omega$ is a $1$-form such that $\omega(d\gamma_s\circ \p t)=0$ for all $(t,s)$ in an open neighbourhood $U \subset \mb{K}^2$ of $(0,0)$, a simple computation shows that $\omega$ must be given in the form
		\[\omega_{(s,t)}=\alpha(t,s)(4t^3\,dX-(3t^2+s)\,dY)\]
	for some $\alpha \in \ms{O}_2$.
	Therefore, $\gamma_s$ does not yield an integral deformation at $s=0$.
\end{example}

\begin{definition}
	Let $f\colon (\mb{K}^n,S) \to (\mb{K}^{n+1},0)$ be a frontal germ.
	An unfolding $F\colon (\mb{K}^n\times\mb{K}^d,S\times\{0\}) \to (\mb{K}^{n+1}\times\mb{K}^d,0)$ of $f$ is \textbf{frontal} if it is frontal as a map germ.
\end{definition}

\begin{theorem}\label{frontal unfoldings and integral deformations}
	Let $f\colon (\mb{K}^n,S) \to (\mb{K}^{n+1},0)$ be a frontal map germ.
	A $d$-parameter unfolding $F=(f_\lambda,\lambda)$ of $f$ is frontal if and only if $\tilde{f_\lambda}$ is an integral deformation of $\tilde f$.
\end{theorem}

\begin{proof}
	Let $F$ be a frontal $d$-parameter unfolding for $f$: there is a $\nu \in \Omega^1(F)$ such that $\nu(dF\circ \eta)=0$ for all $\eta \in \theta_{n+d}$.
	If we set $\nu_0=\nu|_{\lambda=0}$, we can write
		\[\nu_{(x,y,\lambda)}=(\nu_0)_{(x,y)}+\sum^d_{j=1}\lambda_j(\nu_j)_{(x,y,\lambda)}\]
	for some $\nu_1,\dots,\nu_j \in (\mb{K}^n,S) \to T^*\mb{K}^{n+1}$.
	Therefore, $\nu$ may be regarded as a $d$-parameter deformation of $\nu_0$ and the Nash lift of $f_\lambda$,
		\begin{equation}\label{lift frontal unfolding}
			(x,y,\lambda) \mapsto (f_\lambda(x,y),[\nu_{(x,y,\lambda)}])
		\end{equation}
	is an integral $d$-parameter deformation of $f\times [\nu_0]$.
	Since $f\times [\nu_0]$ is an integral map, $\nu_0(df\circ \xi)=0$ for all $\xi \in \theta_n$.
	Properness of $f$ then implies that $f\times [\nu_0]=\tilde f$, and thus the map germ \eqref{lift frontal unfolding} is an integral deformation of $\tilde f$.

	Conversely, let $\tilde f_\lambda$ be an integral deformation of $\tilde f$.
	Taking coordinates $(u,\lambda)$ in the source and Darboux coordinates in the target, the integrability condition becomes
		\[\cvf{u_j}(Y\circ f_\lambda)=(P_1\circ \tilde f_\lambda)\cvf{u_j}(X_1\circ f_\lambda)+\dots+(P_n\circ \tilde f_\lambda)\cvf{u_j}(X_n\circ f_\lambda)\]
	for $j=1,\dots,n$.
	Consider the differential form $\nu \in \Omega^1(F)$ given by
		\[\sum^n_{j=1}(P_j\circ \tilde f_\lambda)\left(dX^j-\sum^d_{k=1}\cvf{\lambda_k}(X_j\circ f_\lambda)\,d\lambda^k\right)-dY+\sum^d_{k=1}\cvf{\lambda_k}(Y\circ f_\lambda)\,d\lambda^k\]
	Using the integrability condition above, we have
		\begin{align*}
			\nu\left(dF\circ\cvf{u_i}\right)=		&\sum^n_{j=1}(P_j\circ \tilde f_\lambda)\frac{\p (X_j\circ f_\lambda)}{\p u_i}-\frac{\p (Y\circ f_\lambda)}{\p u_i}=0;\\
			\nu\left(dF\circ \cvf{\lambda_i}\right)=&\sum^n_{j=1}(P_j\circ \tilde f_\lambda)\left(\frac{\p (X_j\circ \tilde f_\lambda)}{\p \lambda_i}-\frac{\p (X_j\circ \tilde f_\lambda)}{\p \lambda_i}\right)-\\
            &-\frac{\p (Y\circ f_\lambda)}{\p \lambda_i}+\frac{(Y\circ f_\lambda)}{\p \lambda_i}=0
		\end{align*}
	Therefore, $\nu(dF\circ \xi)=0$ for all $\xi \in \theta_{n+d}$ and $F$ is frontal.
\end{proof}

\begin{remark}
	Properness of $f$ is required for the ``if" direction, since $\widetilde{f_u}$ is not guaranteed to be a deformation of $\tilde f$, even if it is integral.
	Nonetheless, the ``only if" direction does not require properness.
\end{remark}

The space of \textbf{infinitesimal integral deformations} of an integral $\tilde f$, defined by Ishikawa in \cite{Ishikawa05}, is given by
\begin{align*}
	\theta_I(\tilde f)=\{v_0(\tilde f_t): \tilde f_t \text{ integral }, \tilde f_0=\tilde f\}; && v_0(\tilde f_t)=\left.\frac{d\tilde f_t}{dt}\right|_{t=0}.
\end{align*}
This space is linear when $\tilde f$ has corank at most $1$ (\cite{Ishikawa05}), but it is known to have a conical structure in higher coranks.
Counterexamples can be constructed using a similar procedure as in \cite{Ishikawa96}.
We also set $T\ms{L}_e\tilde f$ as the subspace of $\theta_I(\tilde f)$ given by those $\tilde f_t$ which are trivial Legendrian deformations of $\tilde f$.

\begin{definition}
	Let $f\colon (\mb{K}^n,S) \to (\mb{K}^{n+1},0)$ be a frontal map germ of integral corank at most $1$.
	We define the space of \textbf{infinitesimal frontal deformations} of $f$ as
		\[\ms{F}(f)=\{v_0(f_t): (t, f_t) \text{ frontal}, f_0=f\}.\]
\end{definition}
As shown in Theorem \ref{legendrian and frontal codimension} below, $\ms{F}(f)$ is the linear projection of $\theta_I(\tilde f)$.
Therefore, if the integral corank of $f$ is at most $1$, $\ms{F}(f)$ is $\mb{K}$-linear; for this reason, any results involving $\ms{F}(f)$ will implicitly assume that $f$ has integral corank at most $1$.
An alternative, direct proof is also given for corank $1$ frontal map germs in Remark \ref{F(f) is linear} below.

\begin{lemma}
	Given a frontal map germ $f\colon (\mb{K}^n,S) \to (\mb{K}^{n+1},0)$, $T\ms{A}_ef \subseteq \ms{F}(f)$.
\end{lemma}

\begin{proof}
	Let $\phi_t\colon (\mb{K}^n,S) \to (\mb{K}^n,S)$, $\psi_t\colon (\mb{K}^{n+1},0) \to (\mb{K}^{n+1},0)$ be two smooth $1$-parameter families of diffeomorphisms and $f_t=\psi_t\circ f\circ \phi^{-1}_t$.
	It is clear by construction that the vector field germ given by $f_t$ is in $T\ms{A}_ef$.
	
	By Theorem \ref{diffeomorphism fixes contactomorphism}, we can lift $\psi_t$ onto a smooth $1$-parameter family $\Psi_t$ of Legendrian diffeomorphisms, in which case we can lift $f_t$ onto an integral deformation $\widetilde{f_t}=\Psi_t\circ \tilde f \circ \phi_t^{-1}$.
	Using Theorem \ref{frontal unfoldings and integral deformations}, we then see that the unfolding $F=(f_t,t)$ is frontal.
	Therefore, the vector field germ given by $f_t$ is in $\ms{F}(f)$, and thus $T\ms{A}_ef \subseteq \ms{F}(f)$.
\end{proof}

\begin{theorem}\label{legendrian and frontal codimension}
	Let $f\colon (\mb{K}^n,0) \to (\mb{K}^{n+1},0)$ be a frontal map germ and $\pi\colon PT^*\mb{K}^{n+1} \to \mb{K}^{n+1}$ be the canonical bundle projection.
	The mapping $t\pi\colon \theta_I(\tilde f) \to \ms{F}(f)$ given by $t\pi(\xi)=d\pi\circ \xi$ is a $\mb{K}$-linear isomorphism and induces an isomorphism
	\begin{equation}\label{eq: frontal codimension}
		\Pi\colon \frac{\ms{F}(f)}{T\ms{A}_ef} \longrightarrow \frac{\theta_I(\tilde f)}{T\ms{L}_e\tilde f}.
	\end{equation}
\end{theorem}

\begin{proof}
	Let $\xi \in \theta_I(\tilde f)$ and $\tilde f_t$ be an integral $1$-parameter deformation of $\tilde f$ and $\xi=v_0(\tilde f_t)$: by Theorem \ref{frontal unfoldings and integral deformations}, $F(t,x)=(t,(\pi \circ \tilde f_t)(x))$ is a frontal $1$-parameter unfolding of $f$.
	Furthermore, using the chain rule, we see that $v_0(\pi\circ \tilde f_t)=t\pi[v_0(\tilde f_t)]$, so $t\pi[\theta_I(\tilde f)] \subseteq \ms{F}(f)$ and $t\pi\colon \theta_I(\tilde f) \to \ms{F}(f)$ is well-defined.
	Conversely, let $\xi \in \ms{F}(f)$ and $(t,f_t)$ be a frontal $1$-parameter deformation of $f$ with $\xi=v_0(f_t)$: by Theorem \ref{frontal unfoldings and integral deformations}, we can lift $f_t$ onto an integral $1$-parameter deformation $\tilde f_t$ of $\tilde f$.
	Using the chain rule, it then follows that $\xi \in t\pi[\theta_I(\tilde f)]$, so $t\pi[\theta_I(\tilde f)]=\ms{F}(f)$.

	We move onto injectivity.
	Let $\tilde f_t(x)=\tilde f(x)+t\tilde h(x,t)$ be an integral $1$-parameter deformation of $\tilde f$ with $(\pi\circ \tilde f_t)(x)=f(x)+th(x,t)$.
	If we assume that $\xi=v_0(\tilde f_t)\in \ker t\pi$, then
		\[0=\left.\frac{df_t}{dt}\right|_{t=0}=[h(x,t)+th_t(x,t)]_{t=0}=h(x,0) \implies h(x,t)=tg(x,t).\]
	Our goal is to show that we can write $\tilde h(x,t)=t\tilde g(x,t)$ for some $\tilde g$, so that $v_0(\tilde f_t)=0$ and thus $\ker t\pi=\{0\}$.

	Since $\tilde f_t$ is an integral deformation of $\tilde f$, it verifies the identity
		\[d(Y\circ f_t)=\sum^n_{j=1}(P_j\circ \tilde f_t)\,d(X_j\circ f_t)\]
	Taking the coefficient of $dx^k$ on both sides of the equation and simplifying yields
		\[t\frac{\p (Y\circ g)}{\p x_k}=\sum^n_{j=1}\left[(P_j\circ \tilde h)\frac{\p (X_j\circ f_t)}{\p x_k}+t(P_j\circ \tilde f)\frac{\p (X_j\circ g)}{\p x_k}\right].\]
	Taking $t=0$ gives us the homogeneous system of equations
		\[0=\sum^n_{j=1}\frac{\p (X_i\circ f)}{\p x_k}(x)(P_j\circ \tilde h)(x,0)\]
	for $k=1,\dots,n$.
	Using the observation from Remark \ref{rank is stored in the X} and the continuity of $P_1\circ \tilde h,\dots,P_n\circ \tilde h$, we conclude that $(P_1\circ \tilde h)(x,0)=\dots=(P_n\circ \tilde h)(x,0)=0$ and thus $\tilde h(x,t)=t\tilde g(x,t)$.

	It only remains to show that $t\pi(T\ms{L}_e\tilde f)=T\ms{A}_ef$.
	Let $\xi \in T\ms{L}_e\tilde f$: there exist $1$-parameter families $\phi_t\colon (\mb{K}^n,S) \to (\mb{K}^n,S)$, $\Psi_t\colon (PT^*\mb{K}^{n+1},w) \to (PT^*\mb{K}^{n+1},w)$ of diffeomorphisms such that $\xi=v_0(\Psi_t\circ \tilde f\circ \phi^{-1}_t)$, with $\Psi_t$ Legendrian.
	Since $\Psi_t$ is Legendrian for all $t$ in a neighbourhood $U \subseteq \mb{K}$ of $0$, there exists a $1$-parameter family $\psi_t\colon (\mb{K}^{n+1},0) \to (\mb{K}^{n+1},0)$ of diffeomorphisms such that $\pi\circ\Psi_t=\psi_t\circ \pi$ for all $t \in U$.
	We then have that $v_0(\psi_t\circ f\circ \phi^{-1}_t)=t\pi[v_0(\Psi_t\circ \tilde f\circ \phi^{-1}_t)]=t\pi(\xi)$, hence $t\pi(\xi) \in T\ms{A}_ef$.

	Conversely, if $\xi \in T\ms{A}_ef$, there exist $1$-parameter families $\phi_t\colon (\mb{K}^n,S) \to (\mb{K}^n,S)$, $\psi_t\colon (\mb{K}^{n+1},0) \to (\mb{K}^{n+1},0)$ of diffeomorphisms such that $\xi=v_0(\psi_t\circ f\circ \phi^{-1}_t)$.
	Using Theorem \ref{diffeomorphism fixes contactomorphism}, there exists a $1$-parameter family of Legendrian diffeomorphisms $\Psi_t\colon (PT^*\mb{K}^{n+1},w) \to (PT^*\mb{K}^{n+1},w)$ such that $\pi\circ \Psi_t=\psi_t\circ\pi$, and thus we can lift $\xi$ onto $v_0(\Psi_t\circ \tilde f\circ \phi^{-1}_t) \in T\ms{L}_e\tilde f$, whose image via $t\pi$ is $\xi$.
\end{proof}

\begin{remark}\label{corank 1 infinitesimal integral deformation}
	Let $f\colon (\mb{K}^n,0) \to (\mb{K}^{n+1},0)$ be a frontal map germ: Theorem \ref{legendrian and frontal codimension} states that $\ms{F}(f)=t\pi[\theta(\tilde f)]$.
	Since $\tilde f$ has corank $1$, a resut by Ishikawa \cite{Ishikawa05} states that
		\[\theta_I(\tilde f)=\{\xi \in \theta(\tilde f): \xi^*\tilde\alpha=0\},\]
	wherein $\tilde \alpha$ denotes the natural lifting of the contact form in $PT^*\mb{K}^{n+1}$.
	Taking Darboux coordinates in $PT^*\mb{K}^{n+1}$,
	\begin{equation}\label{integral corank 1 F(f)}
		\xi\in \ms{F}(f) \iff d\xi_{n+1}-\sum^n_{i=1}(P_i\circ \tilde f)\,d\xi_i \in \ms{O}_n\,d(f^*\ms{O}_n)
	\end{equation}
	In particular, if $f$ has corank $1$ and it is given in prenormal form, Equation \eqref{integral corank 1 F(f)} is equivalent to
		\[\frac{\p \xi_{n+1}}{\p y}-\sum^{n-1}_{j=1}P_j\frac{\p \xi_j}{\p y}+\mu\frac{\p \xi_n}{\p y} \in \ms{O}_n\{p_y\},\]
	where $P_1,\dots,P_{n-1}$ are given as in Equation \eqref{prenormal nash lift}.
\end{remark}

\begin{definition}
	The \textbf{frontal codimension} of $f$ is defined as the dimension of $T^1_{\ms{F}_e}f=\ms{F}(f)/T\ms{A}_ef$.
	We say $f$ is \textbf{$\ms{F}$-finite} or has \textbf{finite frontal codimension} if $\dim T^1_{\ms{F}_e}f < \infty$.
\end{definition}

\subsection{Frontal versality and stability}
In the previous subsection, we formulated the notions of integral deformation and Legendrian codimension purely in terms of frontal unfoldings.
We now show that Ishikawa's results concerning the Legendrian stability and versality of pairs from \cite{Ishikawa05} have a direct parallel in our theory of frontal deformations.

\begin{definition}
	A frontal map germ $f\colon (\mb{K}^n,S) \to (\mb{K}^{n+1},0)$ is \textbf{stable as a frontal} or \textbf{$\ms{F}$-stable} if every frontal unfolding of $f$ is $\ms{A}$-trivial.
\end{definition}

\begin{corollary}\label{trivial Legendrian deformation}
	A frontal map germ $f\colon (\mb{K}^n,S) \to (\mb{K}^{n+1},0)$ is stable as a frontal if and only if $\tilde f$ is Legendrian stable.
\end{corollary}

\begin{proof}
	Assume $f$ is stable as a frontal and let $\tilde f_u$ be an integral deformation of $\tilde f$: by Theorem \ref{frontal unfoldings and integral deformations}, $\tilde f_u$ defines a frontal unfolding $F=(f_u,u)$ of $f$.
	Stability of $f$ then implies that $f_u$ is $\ms{A}$-equivalent to $f$.
	By Corollary \ref{A-equivalence is compatible with frontals}, this then implies that $\tilde f_u$ is Legendrian equivalent to $\tilde f$.
	Since the choice of $\tilde f_u$ was arbitrary, we conclude $\tilde f$ is Legendrian stable.
	The opposite direction is shown similarly.
\end{proof}

\begin{corollary}\label{stable if and only if codimension 0}
	A frontal map germ $f\colon (\mb{K}^n,S) \to (\mb{K}^{n+1},0)$ is $\ms{F}$-stable if and only if its $\ms{F}_e$-codimension is $0$.
\end{corollary}

\begin{proof}
	Corollary \ref{trivial Legendrian deformation} states that $f$ is $\ms{F}$-stable if and only if its Nash lift $\tilde f$ is Legendrian stable.
	Since $f$ has corank at most $1$, so does $\tilde f$, and a result by Ishikawa \cite{Ishikawa05} states that $\tilde f$ is Legendrian stable for the bundle projection $\pi$ if and only if $\theta_I(\tilde f)=T\ms{L}_e\tilde f$.
	However, it follows from Theorem \ref{legendrian and frontal codimension} that this is equivalent to $\ms{F}(f)=T\ms{A}_ef$.
\end{proof}

\begin{example}\label{cusp_umbrella}
	The following frontal hypersurfaces are stable as frontals:
	\begin{enumerate}
		\item Cusp: $X^2-Y^3=0$
		\item Folded Whitney umbrella: $Z^2-X^2Y^3=0$
	\end{enumerate}
\end{example}

Let $f\colon (\mb{K}^n,S) \to (\mb{K}^{n+1},0)$ be a frontal map germ with $d$-parameter unfolding $F=(f_u,u)$, not necessarily frontal.
Recall that the \textbf{pullback} of $F$ by $h\colon (\mb{K}^l,0) \to (\mb{K}^d,0)$ is defined as the $l$-paramter unfolding
	\[(h^*F)(x,v)=(f_{h(v)}(x),v)\]

\begin{definition}
	Let $f\colon (\mb{K}^n,S) \to (\mb{K}^{n+1},0)$ be a frontal map germ.
	A frontal unfolding $F$ of $f$ is \textbf{$\ms{F}$-versal} or \textbf{versal as a frontal} if, given any other frontal unfolding $G$ of $f$, there exist unfoldings $T\colon (\mb{K}^{n+1}\times\mb{K}^d,0) \to (\mb{K}^{n+1}\times\mb{K}^d,0)$ and $\Sigma\colon(\mb{K}^n\times\mb{K}^d,S\times\{0\}) \to (\mb{K}^n\times\mb{K}^d,S\times\{0\})$ of the identity such that
			\[G=T\circ h^*F \circ \Sigma\]
	for some map germ $h$.
\end{definition}

\begin{lemma}\label{frontal and integral versality}
	Given a frontal map germ $f\colon (\mb{K}^n,S) \to (\mb{K}^{n+1},0)$, a frontal unfolding $F=(f_u,u)$ is $\ms{F}$-versal if and only if $\tilde f_u$ is a Legendre versal deformation of $\tilde f$.
\end{lemma}

\begin{proof}
	Assume $F$ is a versal frontal unfolding of $f$ and let $(\widetilde{g_u})$ be an $s$-parameter integral deformation of $\tilde f$.
	Theorem \ref{frontal unfoldings and integral deformations} implies that the $s$-parameter unfolding $G=(u,g_u)$ is frontal.
	By versality of $F$, there exist unfoldings $\mc{T}\colon (\mb{K}^{n+1}\times\mb{K}^d,0) \to (\mb{K}^{n+1}\times\mb{K}^d,0)$, $\mc{S}\colon (\mb{K}^n\times\mb{K}^d,S\times\{0\}) \to (\mb{K}^n\times\mb{K}^d,S\times\{0\})$ of the identity map germ and a smooth map germ $h\colon (\mb{K}^s,0) \to (\mb{K}^d,0)$ such that $G=\mc{T}\circ h^*F \circ \mc{S}^{-1}$.

	Let $f\colon N \to Z$ be a representative of $f$ which is a proper frontal map, and $F\colon \mc{N} \to \mc{Z}$ be a representative of $F$ such that $\mc{N} \subseteq N\times \mb{K}^d$.
	A simple computation shows that $\Sigma(F)=\Sigma(f)\times\{0\}$; therefore, since $\Sigma(f)$ is nowhere dense in $N$, $\Sigma(F)$ is nowhere dense in $\mc{N}$ and $F$ is a proper frontal map.
	Theorem \ref{frontal unfoldings and integral deformations} then states that $f_u$ lifts into integral deformation of $\tilde f$.
	
	Now consider representatives $h^*F=(u,f_{h(u)})\colon \mc{N}_1 \to \mc{Z}_1$, $\mc{S}=(u,\sigma_u)\colon \mc{N}_1 \to \mc{N}_2$, $\mc{T}=(u,\tau_u)\colon \mc{Z}_1 \to \mc{Z}_2$ and $G\colon \mc{N}_2 \to \mc{Z}_2$ such that $G=\mc{T}\circ h^*F\circ \mc{S}^{-1}$ as mappings.
	Since $(\tau_u)$ is a smooth $d$-parameter family of diffeomorphisms, we can lift it onto a $d$-parameter family of smooth Legendrian diffeomorphisms $T_u\colon PT^*\mc{Z}_1 \to PT^*\mc{Z}_2$.
	Therefore,
		\[\widetilde{g_u}=T_u\circ \widetilde{f_{h(u)}}\circ \sigma_u^{-1}\]
	and $\widetilde{f_u}$ is a versal Legendrian deformation of $\tilde f$.

	Conversely, let $\tilde f_u$ be a versal integral deformation of $\tilde f$ and $G=(g_u,u)$ be a frontal $s$-parameter unfolding of $f$.
	Theorem \ref{frontal unfoldings and integral deformations} implies that the $s$-parameter deformation $\widetilde{g_u}$ is integral.
	By versality of $\tilde f_u$, there exist smooth families of diffeomorphisms $T_u\colon (PT^*\mb{K}^{n+1},w) \to (\mb{K}^{n+1},w)$ and $\sigma_u\colon (\mb{K}^n,S) \to (\mb{K}^n,S)$ and a smooth map germ $h\colon (\mb{K}^s,0) \to (\mb{K}^d,0)$ verifying the following:
	\begin{enumerate}
		\item $T_u$ is a Legendrian diffeomorphism for all $u$; \label{Tu Legendrian}
		\item $T_0$ and $\sigma_0$ are the identity map germs;
		\item $\widetilde{g_u}=T_u\circ \tilde f_{h(u)}\circ \sigma_u$.
	\end{enumerate}
	By Item \ref{Tu Legendrian}, we can find a smooth family of diffeomorphisms $\tau_u\colon (\mb{K}^{n+1},0)\to (\mb{K}^{n+1},0)$ such that $\pi\circ T_u=\tau_u\circ \pi$ and $\tau_0$ is the identity map germ.
	It follows that
		\[\widetilde{g_u}=T_u\circ \tilde f_{h(u)}\circ \sigma_u \iff g_u=\tau_u\circ f_{h(u)}\circ \sigma_u.\]
	If we now consider the unfoldings $\mc{T}=(\tau_u,u)$ and $\mc{S}=(\sigma_u,u)$, we have $G=\mc{T}\circ h^*F\circ \mc{S}$.
	We conclude that $F$ is versal as a frontal.
\end{proof}

\begin{theorem}[Frontal versality theorem]\label{frontal versality}
	Given a frontal map germ $f\colon (\mb{K}^n,S) \to (\mb{K}^{n+1},0)$,
	\begin{enumerate}
	\item $f$ admits a frontal versal unfolding if and only if it is $\ms{F}$-finite;
	\item a frontal unfolding $F(u,x)=(u,f_u(x))$ of $f$ is versal as a frontal if and only if
		\begin{align*}
			\ms{F}(f)=T\ms{A}_e f+\Sp_{\mb{K}}\{\dot F_1,\dots,\dot F_d\}, && \dot F_j=\left.\frac{\p f_u}{\p u_j}\right|_{u=0}.
		\end{align*}
	\end{enumerate}
\end{theorem}

To show Theorem \ref{frontal versality}, we shall make use of

\begin{theorem}[Ishikawa's Legendre versality theorem \cite{Ishikawa05}]\label{integral versality}
	Given an integral $\tilde f\colon (\mb{K}^n,S) \to (PT^*\mb{K}^{n+1},w)$ of corank at most $1$,
		\begin{enumerate}
		\item $\tilde f$ admits a versal Legendrian unfolding if and only if its Legendrian codimension is finite;
		\item a Legendrian unfolding $\tilde f_u$ of $\tilde f$ is versal if and only if
		\begin{align}\label{legendrian versality equation}
			\theta_I(\tilde f)=T\ms{L}_e \tilde f+\Sp_{\mb{K}}\left\{\left.\frac{\p \tilde f_u}{\p u_1}\right|_{u=0},\dots,\left.\frac{\p \tilde f_u}{\p u_d}\right|_{u=0}\right\}.
		\end{align}
		\end{enumerate}
\end{theorem}

\begin{proof}[Proof of Theorem \ref{frontal versality}]
	By Lemma \ref{frontal and integral versality}, a frontal unfolding $F=(f_u,u)$ of $f$ is versal as a frontal if and only if the smooth family $\widetilde{f_u}$ is a versal Legendre deformation of $\tilde f$.
	In particular, it follows from Theorem \ref{frontal unfoldings and integral deformations} that $\tilde f$ admits a versal Legendrian deformation if and only if $f$ admits a versal frontal unfolding.
	This fact shall be used to prove both items.
	
	By Theorem \ref{integral versality}, $f$ admits a $\ms{F}$-versal unfolding if and only if $\tilde f$ has finite Legendre codimension.
	However, it was proved in Theorem \ref{legendrian and frontal codimension} that this is equivalent to $f$ being $\ms{F}$-finite.
	This shows the first Item.

	We move onto the second Item.
	If $F$ is $\ms{F}$-versal, $\tilde f_u$ is a Legendre versal unfolding of $\tilde f$ by Lemma \ref{frontal and integral versality} and Equation \eqref{legendrian versality equation} holds.
	Computing the image via $t\pi$ on both sides of Equation \eqref{legendrian versality equation} and using Theorem \ref{legendrian and frontal codimension}, we get
	\begin{multline}\label{frontal versality equation}
		\ms{F}(f)=T\ms{A}_e\tilde f+t\pi\left[\Sp_{\mb{K}}\left\{\left.\frac{\p \tilde f_u}{\p u_1}\right|_{u=0},\dots,\left.\frac{\p \tilde f_u}{\p u_d}\right|_{u=0}\right\}\right]=\\
				=T\ms{A}_ef+\Sp_{\mb{K}}\{\dot F_1,\dots,\dot F_d\}.
	\end{multline}
	Conversely, let us assume that \eqref{frontal versality equation} holds: using Theorem \ref{legendrian and frontal codimension}, we see that \eqref{legendrian versality equation} holds as well.
	Therefore, $F$ is versal as a frontal.
	This shows the second Item.
\end{proof}

\section{A geometric criterion for $\ms{F}$-finiteness}\label{mathergaffney}
The Mather-Gaffney criterion states that a smooth $f\colon (\mb{C}^n,S) \to (\mb{C}^{n+1},0)$ is $\ms{A}$-finite if and only if there is a finite representative $f\colon N \to Z$ with isolated instability.
For example, the generic singularities for $n=2$ are transversal double points, with Whitney umbrellas and triple points in the accumulation (see e.g. \cite{MondNuno_Singularities} \S 4.7).
This implies that generic frontal singularities such as the folded Whitney umbrella (see Example \ref{cusp_umbrella}) are not $\ms{A}$-finite, since it contains cuspidal edges near the origin.
Nonetheless, cuspidal edges are generic within the subspace of frontal map germs $(\mb{C}^2,S) \to (\mb{C}^3,0)$ (\cite{Arnold_CWF}), which suggests the existence of a Mather-Gaffney-type criterion for frontal hypersurfaces.

\begin{proposition}\label{F finite plane curve}
	A germ of analytic plane curve $\gamma\colon (\mb{C},S) \to (\mb{C}^2,0)$ is $\ms{F}$-finite if and only if it is $\ms{A}$-finite.
\end{proposition}

\begin{proof}
	If $\gamma$ is $\ms{A}$-finite, it is clear that it is also $\ms{F}$-finite, since
		\[\ms{F}(\gamma) \subseteq \theta(\gamma) \implies \dim\frac{\ms{F}(\gamma)}{T\ms{A}_e\gamma} \leq \dim\frac{\theta(\gamma)}{T\ms{A}_e\gamma} < \infty\]
	Assume $\gamma$ is $\ms{F}$-finite, and let $\gamma\colon N \to Z$ be a representative of $\gamma$.
	By the Curve Selection Lemma \cite{BurVerona72}, $\Sigma(\gamma)$ is an isolated subset in $N$, so we can assume (by shrinking $N$ if necessary) that $\gamma(N\backslash S)$ is a smooth submanifold of $Z$ and $\gamma^{-1}(\{0\})=S$.
	By the Mather-Gaffney criterion, it then follows that $\gamma$ is $\ms{A}$-finite, as stated.
\end{proof}

Given a frontal map $f\colon N \to Z$ and $z \in Z$, let $f_z\colon (N,f^{-1}(z))\to (Z,z)$.
We define $\ms{F}(f)$ as the sheaf of $\ms{O}_Z$-modules given by the stalk $\ms{F}(f)_z=\ms{F}(f_z)$.
We also set $\theta_N$ (resp. $\theta_Z$) as the sheaf of vector fields on $N$ (resp. $Z$) and the quotient sheaves
    \begin{align*}
        \ms{T}^1_{\ms{R}_e}f=\frac{\ms{F}(f)}{tf(\theta_N)}; && \ms{T}^1_{\ms{F}_e}f=\frac{f_*\left(\ms{T}^1_{\ms{R}_e}f\right)}{\omega f(\theta_Z)};
    \end{align*}

\begin{remark} \label{R tilde f finito}
	If $f$ is finite, we can take coordinates in $N$ and $W$ such that $\tilde f(x,y)=(x,f_n(x,y),\dots,f_{2n+1}(x,y))$.
	By \cite{Ishikawa94}, we have the identity
		\[R_{\tilde f}:=\left\{\lambda \in \ms{O}_N\colon d\lambda \in \ms{O}_N\,d\left({\tilde f}^*\ms{O}_W\right)\right\}=\left(\frac{\p}{\p y}\right)^{-1}\ms{O}_N\left\{\frac{\p \tilde f_n}{\p y},\dots,\frac{\p \tilde f_{2n+1}}{\p y}\right\}\]
	which is a $\ms{O}_N$-finite algebra by \cite{Ishikawa83}.
	Since $f$ is finite, $R_{\tilde f}$ is $\ms{O}_Z$-finite.
\end{remark}

\begin{proposition}[\cite{Ishikawa05}]\label{theta integral isomorfismo}
	Let $f\colon (\mb{C}^n,S) \to (\mb{C}^{n+1},0)$ be a frontal map germ.
	If $\tilde f$ is $\ms{A}$-equivalent to an analytic $g\colon (\mb{C}^n,S) \to (\mb{C}^{2n+1},0)$ (not necessarily integral) such that $\codim_{\mb{C}} \Sigma(g) > 1$,
		\[\frac{\theta_I(\tilde f)}{T\ms{L}_e\tilde f}\cong_{\ms{O}_Z}\frac{R_{\tilde f}}{\ms{O}_Z\{1,\tilde p_1,\dots,\tilde p_n\}}\]
	where $\tilde p_1,\dots, \tilde p_n$ are the coordinates of $\tilde f$ in the fibres of $\pi$.
\end{proposition}

\begin{remark}\label{singular set nash lift}
	Let $f$ and $\tilde f$ be given as in the statement above.
	If we assume that $f$ has corank $1$ and is given as in Equation \eqref{eq: prenormal corank 1}, $\Sigma(\tilde f)=V(p_y,\mu_y)$.
\end{remark}

\begin{corollary}\label{frontal sheaf coherent}
	Let $f\colon (\mb{C}^n,S) \to (\mb{C}^{n+1},0)$ be a frontal map germ.
	If $f$ is finite and $\codim V(p_y,\lambda_y) > 1$, there is a representative $f\colon N \to Z$ of $f$ such that $\ms{T}_{\ms{F}_e}^1f$ is a coherent sheaf.
\end{corollary}

\begin{proof}
	Using Proposition \ref{theta integral isomorfismo}, we have
		\[\frac{R_{\tilde f_w}}{\ms{O}_Z\{1,\tilde p_1,\dots,\tilde p_n\}}\cong_{\ms{O}_Z}\frac{\theta_I(\tilde f_w)}{T\ms{L}_e\tilde f_w}=(\ms{T}_{\ms{F}_e}^1f)_{\pi(w)}\]
	Since $f$ is finite, $R_{\tilde f_w}$ is $\ms{O}_{Z,\pi(w)}$-finite, as shown in Remark \ref{R tilde f finito}.
	Therefore, the stalk of $\ms{T}_{\ms{F}_e}^1f$ at $\pi(w)$ is finitely generated and $\ms{T}_{\ms{F}_e}^1f$ is of finite type.

	Let $V \subset Z$ be an open set and $\beta\colon\ms{O}^q_{Z\upharpoonright V} \to \left(\ms{T}_{\ms{F}_e}^1f\right)_{\upharpoonright V}$ an epimorphism of $\ms{O}_Z$-modules.
	Since $\ms{O}_Z$ is a Noetherian ring, every submodule of $\ms{O}^q_{Z\upharpoonright V}$ is finitely generated.
	In particular, $\ker \beta$ is finitely generated.
	We then conclude that $\ms{T}_{\ms{F}_e}^1f$ is a coherent sheaf.
\end{proof}

\begin{theorem}[Mather-Gaffney criterion for frontal maps]\label{frontal mather gaffney}
	\label{thm: frontal mather gaffney}
	Let $f\colon (\mb{C}^n,S) \to (\mb{C}^{n+1},0)$ be a frontal map germ.
	If $f$ is finite and $\codim_\mb{C} \Sigma(\tilde f) > 1$, $f$ is $\ms{F}$-finite if and only if there exists a representative $f\colon N' \to Z'$ of $f$ such that the restriction $f\colon N' \backslash S \to Z'\backslash \{0\}$ is locally $\ms{F}$-stable.
\end{theorem}

\begin{proof}
	The case for $n=1$ follows easily from the Mather-Gaffney criterion for $\ms{A}$-equivalence and Proposition \ref{F finite plane curve}.
	Therefore, we assume $n > 1$.

	Suppose first that $f$ has finite $\ms{F}$-codimension: by Corollary \ref{frontal sheaf coherent}, $\ms{T}_{\ms{F}_e}^1f$ is a coherent sheaf.
	In addition,
		\[\dim_{\mb{C}}(\ms{T}_{\ms{F}_e}^1f)_0=\dim_{\mb{C}}T_{\ms{F}_e}^1f=\codim_{\ms{F}_e}f < \infty\]
	By Rückert's Nullstellensatz, there exists an open neighbourhood $Z'$ of $0$ in $Z$ such that $\supp \ms{T}_{\ms{F}_e}^1f\cap Z \subseteq \{0\}$.
	Therefore, every other stalk of $\ms{T}_{\ms{F}_e}^1f$ is $0$, and the restriction of $f$ to $N'\backslash\{0\}$ is $\ms{F}$-stable, where $N'=f^{-1}(Z')$.

	Conversely, suppose that there exists a representative $f\colon N' \to Z'$ such that the restriction $f\colon N' \backslash \{0\} \to Z'\backslash \{0\}$ is locally $\ms{F}$-stable.
	Given $z \in Z\backslash \{0\}$, $(\ms{T}_{\ms{F}_e}^1f)_z=0$, so there exists an open neighbourhood $U$ of $0$ in $Z$ such that $\supp \ms{T}_{\ms{F}_e}^1f\cap U \subseteq \{0\}$.
	By Rückert's Nullstellensatz, it follows that the dimension of the stalk of $\ms{T}_{\ms{F}_e}^1f$ at $0$ is finite, but that dimension is equal to $\codim_{\ms{F}_e}f$.
	We conclude that the germ of $f$ at $0$ is $\ms{F}$-finite.
\end{proof}

\section{Frontal reduction of a corank $1$ map germ}\label{frontal reductions}
In \cite{FrontalSurfaces}, we presented the notion of frontalisation for a fold surface $f \colon (\mb{C}^2,S) \to (\mb{C}^3,0)$, and proved that the frontalisation process preserves some of the topological invariants of $f$.
We also defined frontal versions of Mond's $S_k$, $B_k$, $C_k$ and $F_4$ singularities (see \cite{Mond_Classification}), observing that none of them are wave fronts.
We now seek to describe a more general procedure to generate frontals using arbitrary corank $1$ map germs.

\begin{example}
    Let $\gamma\colon (\mb{K},0) \to (\mb{K}^2,0)$ be the parametrised curve $\gamma(t)=(t^3,t^4)$: the unfolding $\Gamma\colon (\mb{K}^3\times\mb{K},0) \to (\mb{K}^3\times\mb{K}^2,0)$ given by
        \[\Gamma(u,t)=(u,t^3+u_1t,t^4+u_2t+u_3t^2)=(u,p(u,t),q(u,t))\]
    is an $\ms{A}$-miniversal deformation for $\gamma$.
    By Proposition \ref{pq criterion frontality} and since $\deg_t p_t < \deg_t q_t$, $\Gamma$ is frontal if and only if $p_t|q_t$.
    If $\mu \in \ms{O}_1$ is such that $q_t=\mu p_t$, a simple computation then shows that the identity
		\[4t^3+u_2+2u_3t=(3t^2+u_1)(\mu_1t+\mu_0)\]
    holds if and only if $u_2=\mu_0=0$, $\mu_1=4/3$ and $2u_3=3u_1$.
    Setting $h(v)=(3v,0,2v)$, we obtain the unfolding
    	\[h^*\Gamma(t,v)=(v,t^3+3vt,t^4+2vt^2)\]
    which is a swallowtail singularity.
\end{example}

In this section, we show that the frontal reduction of the versal unfolding of a plane curve is a $\ms{F}$-versal unfolding.
The proof of this result gives a procedure to compute the frontal reduction of a given unfolding (versal or otherwise) via a system of polynomial equations, which may be solved using a computer algebra system such as \textsc{Oscar} or \textsc{Singular}.

\begin{remark}[Piuseux parametrisation]\label{Piuseux theorem}
	Let $\gamma\colon (\mb{C},0) \to (\mb{C}^2,0)$ be an analytic plane curve with isolated singularities.
	There exists a $f \in \mb{C}\{x,y\}$ such that $f\circ \gamma=0$.
	By Piuseux's Theorem (see e.g. \cite{Wall}, Theorem 2.2.6, or \cite{deJongPfister}, Theorem 5.1.1), if $\alpha=\ord f$, $f(t^\alpha,t^{\alpha+1}h(t))=0$ for some $h \in \mb{C}\{t\}$.
	Therefore, $\gamma$ is $\ms{A}$-equivalent to the plane curve
		\[t \mapsto (t^\alpha,t^{\alpha+1}g(t)).\]
	In particular, $\gamma$ is $\ms{A}$-finite (and thus finitely determined) by the Mather-Gaffney criterion, so we can further assume that $g \in \mb{C}[t]$.

	If $\mb{K}=\mb{R}$, it suffices to replace $\gamma$ with its complexification $\gamma_{\mb{C}}$ in the argument above, as $\gamma$ is analytic.
	Therefore, such a parametrisation also exists in the real case.
\end{remark}

\begin{lemma}\label{versal plane curve}
	Let $\gamma\colon (\mb{K},0) \to (\mb{K}^2,0)$ be the plane curve from Remark \ref{Piuseux theorem}.
	There exists a smooth $d$-parameter deformation $(g_w)$ such that
		\[\Gamma(u,v,w,t)=\left(u,v,w,t^\alpha+\sum^{\alpha-2}_{j=1}u_jt^j,\sum^{\alpha-1}_{j=1}v_jt^j+t^{\alpha+1}g_w(t)\right)\]
	is a miniversal unfolding of $\gamma$.
\end{lemma}

\begin{proof}
	Let $G=\{g_1,\dots,g_d\} \subset \mb{K}[t]$ be a $\mb{K}$-basis for $T^1_{\ms{A}_e}\gamma$: by Martinet's theorem, a miniversal unfolding for $\gamma$ is given by the expression
	\begin{equation}\label{Martinet versal curve}
		\Gamma(x,t)=(x,\gamma(t)+x_1g_1(t)+\dots+x_{n-1}g_{n-1}(t))
	\end{equation}
	A simple computation shows that
	\begin{align}\label{A tangente curva plana}
		T\ms{A}_e\gamma\subseteq
			\ms{O}_1\left\{
			\begin{pmatrix}
				\alpha t^{\alpha-1}\\
				(\alpha+1)t^\alpha q_0(t)+t^{\alpha+1}q_0'(t)
			\end{pmatrix}
			\right\}+\mf{m}_1^\alpha\ms{O}_1^2.
	\end{align}
	Using Equation \eqref{A tangente curva plana}, we may assume that $g_j(t)=(t^j,0)$ and $g_{j+\alpha-2}(t)=(0,t^j)$ for $1 \leq j \leq \alpha-2$.
	Setting $g_w(t)=g(t)+w_1g_{2\alpha-1}(t)+\dots+w_{d-2\alpha+1}g_d(t)$, Equation \eqref{Martinet versal curve} becomes
	\begin{equation*}
		\Gamma(u,v,w,t)=\left(u,v,w,t^\alpha+\sum^{\alpha-2}_{j=1}u_jt^j,t^{\alpha+1}g_w(t)+\sum^{\alpha-1}_{j=1}v_jt^j\right),
	\end{equation*}
	as claimed.
\end{proof}

\begin{remark}\label{pullback polynomial}
	Let $h\colon (\mb{K}^r,0) \to (\mb{K}^d,0)$ be a smooth map-germ and $\Gamma$ be the unfolding from Lemma \ref{versal plane curve}.
	The pullback $h^*\Gamma$ is given by
		\[(h^*\Gamma)(x,t)=\left(x,t^\alpha+\sum^{\alpha-2}_{j=1}u_j(x)t^j,\sum^{\alpha-1}_{j=1}v_j(x)t^j+t^{\alpha+1}g_{w(x)}(t)\right),\]
	where $u_j(x)\equiv(u_j\circ h)(x)$, $v_j(x)\equiv(v_j\circ h)(x)$ and $w(x)\equiv(w\circ h)(x)$.
	As we saw in the proof of Lemma \ref{versal plane curve},
		\[g_w(t)=g(t)+w_1g_{2\alpha-1}(t)+\dots+w_{d-2\alpha+1}g_d(t),\]
	where $g$ can be assumed to be a polynomial function (due to Remark \ref{Piuseux theorem}).
	Therefore, the component functions of $h^*\Gamma$ are elements of $\ms{O}_r[t]$, the algebra of polynomials on $t$ with coefficients in $\ms{O}_r$.
\end{remark}

\begin{theorem}\label{existence of frontal reduction}
	If $\gamma$ has a miniversal $d$-parameter unfolding $\Gamma$, there is an immersion $h\colon (\mb{K}^l,0) \to (\mb{K}^d,0)$ with the following properties:
	\begin{enumerate}
		\item $h^*\Gamma$ is a frontal unfolding of $\gamma$ \label{frontal reduction 1};
		\item if $(h')^*\Gamma$ is frontal for any other $h'\colon (\mb{K}^{l'},0) \to (\mb{K}^d,0)$, $(h')^*\Gamma$ is equivalent as an unfolding to a pullback of $h^*\Gamma$ \label{frontal reduction 2}.
	\end{enumerate}
	Therefore, $h^*\Gamma$ is a frontal miniversal unfolding.
\end{theorem}

We shall denote $h^*\Gamma$ as $\Gamma_\ms{F}$ and call it a \emph{frontal reduction} of $\Gamma$.

\begin{proof}
	Let $\Gamma$ be the unfolding from Lemma \ref{versal plane curve} and $d=\codim_{\ms{A}_e}\gamma$.
	We first want to show that there is an immersion $h\colon (\mb{K}^\ell,0) \to (\mb{K}^d,0)$ making $h^*\Gamma$ a frontal map germ; to do so, we shall derive a system of equations that determines whether a given pullback yields a frontal unfolding.

	Let $(h^*\Gamma)(x,t)=(x,P(x,t),Q(x,t))$.
	By Remark \ref{pullback polynomial}, $Q \in \ms{O}_r[t]$, so we can write $Q(x,t)=q_1(x)t+\dots+q_\beta(x)t^\beta$.
	Since $h^*\Gamma$ is a corank $1$ map germ, Corollary \ref{pq criterion frontality} states that it is frontal if and only if either $P_t|Q_t$ or $Q_t|P_t$; in particular, we can assume that $\deg_t P_t \leq \deg_t Q_t$, allowing us to impose the condition $P_t|Q_t$ to $h^*\Gamma$.

	If $Q_t=\mu P_t$ for some $\mu \in \ms{O}_{r+1}$, there will exist $\mu_0,\dots,\mu_{\beta-\alpha}$ such that $\mu(x,t)=\mu_0(x)+\dots+\mu_{\beta-\alpha}(x)t^{\beta-\alpha}$.
	Therefore, the identity $Q_t=\mu P_t$ is equivalent to
	\begin{equation}\label{frontal condition unfolding}
		kq_k(x)=\sum_{i+j=k}iu_i(x)\mu_j(x)
	\end{equation}
	for $k=1,2\dots,\beta$.
	For $k \geq \alpha$, we may solve for $\mu_{k-\alpha}$ to get the expression
	\begin{align*}
		\mu_{k-\alpha}(x)=\frac{k}{\alpha}q_k(x)-\frac{1}{\alpha}\sum_{i+j=k}iu_i(x)\mu_j(x); && u_\alpha(x)\equiv 1.
	\end{align*}
	The remaining terms define an immersion germ $h\colon (\mb{K}^{d-\alpha+1},0)\to (\mb{K}^d,0)$ given by $h(u,w)=(u,v(u,w),w),$ which verifies Equation \eqref{frontal condition unfolding} by construction.
	This proves Item \ref{frontal reduction 1}.

	Let $\Lambda$ be a frontal unfolding of $\gamma$: versality of $\Gamma$ implies that $\Lambda$ is equivalent to $(h')^*\Gamma$ for some $h'\colon (\mb{K}^r,0) \to (\mb{K}^d,0)$.
	Let $h\colon V \to U$ be a one-to-one representative of $h$, $\pi\colon U \to V$ be the projection
		\[\pi(x_1,\dots,x_d)=(x_1,\dots,x_{\alpha-2}, x_{2\alpha-2}, \dots, x_d)\]
	and $h'\colon V' \to U'$ be a representative of $h'$.
	Since $(h')^*\Gamma$ is frontal, $h'$ verifies Equation \eqref{frontal condition unfolding} and thus $h'(V') \subseteq h(V)$ by construction.
	Given $v' \in V'$, there exists a unique $v \in V$ such that
		\[h'(v')=h(v) \implies (\pi\circ h')(v')=v \implies (h\circ \pi\circ h')(v')=h(v)=h'(v'),\]
	and thus $(h')^*\Gamma=(h\circ \pi\circ h')^*\Gamma=(\pi\circ h')^*(h^*\Gamma)$.
\end{proof}

\begin{example}
	Consider Arnol'd's $E_8$ singularity, $\gamma(t)=(t^3,t^5)$.
	A versal unfolding of this curve is given by
		\[(u,v,w,t^3+ut,t^5+wt^4+v_2t^2+v_1t)=(u,v,w,p(u,t),q(v,w,t))\]
	The frontal reduction of this unfolding may now be computed using Equation \eqref{frontal condition unfolding}, which can be written in matrix form as
		\[\begin{pmatrix}
			5 \\ 4w \\ 0 \\ 2v_2 \\ v_1
		\end{pmatrix}
		=
		\begin{pmatrix}
			3 & 0 & 0 \\
			0 & 3 & 0 \\
			u & 0 & 3 \\
			0 & u & 0 \\
			0 & 0 & u
		\end{pmatrix}
		\begin{pmatrix}
			\mu_2 \\ \mu_1 \\ \mu_0
		\end{pmatrix}
		\implies
		\begin{pmatrix}
			\mu_2 \\ \mu_1 \\ \mu_0
		\end{pmatrix}
		=
		\frac{1}{9}
		\begin{pmatrix}
			15 \\ 12w \\ -5u
		\end{pmatrix}
		\]
	Since this system has five equations and only three unknowns, we can now solve for $v$, yielding $v_1=-5/9u^2$ and $v_2=2/3w$.
\end{example}

\begin{remark}\label{reducing stable unfoldings}
	While the method of frontal reductions successfully turns $\ms{A}$-versal unfoldings into $\ms{F}$-versal unfoldings, the same does not hold for stable unfoldings.
	For example, given the plane curve $\gamma(t)=(t^2,t^{2k+1})$, $k > 1$, a stable unfolding of $\gamma$ is given by $f(u,t)=(u,t^2,t^{2k+1}+ut)$.
	However, the only pullback that can turn $f$ into a frontal map germ is $u(s)=0$, giving us $\gamma$, which is not stable by hypothesis.

	A more general method to compute stable unfoldings will be given in \S \ref{stable unfoldings}.
\end{remark}

\begin{corollary}\label{formula for frontal codimension}
	Given $\gamma\colon (\mb{K},0) \to (\mb{K}^2,0)$,
		\[\codim_{\ms{F}_e}\gamma=\codim_{\ms{A}_e}\gamma-\mult(\gamma)+1\]
	Consequently, if $\gamma(\mb{K},0)$ is the zero locus of some analytic $g \in \ms{O}_2$,
		\[\codim_{\ms{F}_e}\gamma=\tau(g)-\ord(g)-\frac{1}{2}\mu(g)+1.\]
\end{corollary}

\begin{proof}
	In the proof of Theorem \ref{existence of frontal reduction}, we see that $l=d-\alpha+1$, where $d=\codim_{\ms{A}_e}\gamma$ and $\alpha=\mult(\gamma)$.
	Since $h^*\Gamma$ is a miniversal $l$-parameter unfolding, $\codim_{\ms{F}_e}\gamma=l$, giving the first identity.

	Now assume $\mb{K}=\mb{C}$: Milnor's formula \cite{Milnor_RedBook} states that the delta invariant $\delta(g)$ and the Milnor number $\mu(g)$ of $g$ are related via the identity $2\delta(g)=\mu(g)$, since $\gamma$ is a mono-germ.
	On the other hand, a result in \cite{GreuelLossenShustin} states that $\codim_{\ms{A}_e}\gamma=\tau(g)-\delta(g)=\tau(g)-1/2\mu(g)$, $\tau$ being the Tjurina number, hence yielding the expression
		\[\codim_{\ms{F}_e}\gamma=\tau(g)-\frac{1}{2}\mu(g)-\mult(\gamma)+1\]
	In particular, the order of $g$ is equal to $\mult(\gamma)$ (see \cite{deJongPfister} Corollary 5.1.6).

	For $\mb{K}=\mb{R}$, simply note that $\mu(g)=\mu(g_\mb{C})$, $\ord(g)=\ord(g_\mb{C})$ and $\tau(g)=\tau(g_\mb{C})$, where $g_\mb{C}$ is the complexification of $g$.
\end{proof}

\begin{example}
	Let $\gamma\colon (\mb{C},0) \to (\mb{C}^2,0)$ be the $A_{2k}$ singularity, with normalisation $\gamma(t)=(t^2,t^{2k+1})$.
	Direct computations show that
	\begin{align*}
		\frac{\theta(\gamma)}{T\ms{A}_e\gamma}\cong \Sp\{(0,t^{2\ell+1}): 0 \leq \ell < k\}; && \frac{\ms{F}(\gamma)}{T\ms{A}_e\gamma}\cong \Sp\{(0,t^{2\ell+1}): 1 \leq \ell < k\},
	\end{align*}
	from which follows that its $\ms{A}_e$-codimension is $k$ and its $\ms{F}_e$-codimension is $k-1$.
	Therefore, we have $\codim_{\ms{F}_e}\gamma=k-1=k-2+1=\codim_{\ms{A}_e}\gamma-\mult(\gamma)+1$, as expected.

	The image of $\gamma$ is given as the zero locus of the function $g(x,y)=y^2-x^{2k+1}$.
	Using the second expression for the frontal codimension, we have
		\[\tau(g)-\frac{1}{2}\mu(g)=\codim_{\ms{F}_e}\gamma+\ord(g)-1=k-1+2-1=k\]
	as expected, since both the Tjurina and Milnor numbers of $g$ are $2k$.
\end{example}

In \cite{FrontalSurfaces} \S5, we introduced the notion of frontal Milnor number $\mu_{\ms{F}}$ for a frontal multi-germ $f\colon (\mb{C}^n,S) \to (\mb{C}^{n+1},0)$.
This analytic invariant was defined in a similar fashion to Mond's image Milnor number \cite{Mond_Vanish}, only changing smooth stabilisations for frontal ones.
We then conjectured that $\mu_\ms{F}$ verified an adapted version of Mond's conjecture, which we called \emph{Mond's frontal conjecture}.

Applying \cite{FrontalSurfaces}, Proposition 5.10 to Corollary \ref{formula for frontal codimension}, we can now prove Mond's frontal conjecture in dimension $1$.

\begin{corollary}
	Given a plane curve $\gamma\colon (\mb{K},S) \to (\mb{K}^2,0)$, $\mu_{\ms{F}}(\gamma) \geq \codim_{\ms{F}}(\gamma)$, with equality if $\gamma$ is quasi-homogeneous.
\end{corollary}

\begin{proof}
	Let $\gamma$ be a non-constant analytic plane curve.
	By the Curve Selection Lemma \cite{BurVerona72}, $\gamma$ has an isolated singularity at the origin, so it is $\ms{A}$-finite and
		\[\mu_I(\gamma) \geq \codim_{\ms{A}_e}(\gamma),\]
	with equality if $\gamma$ is quasi-homogeneous (see \cite{Mond_Vanish}).
	By Corollary \ref{formula for frontal codimension}, $\gamma$ is $\ms{F}$-finite and $\codim_{\ms{A}_e}(\gamma)=\codim_{\ms{F}_e}(\gamma)+\mult(\gamma)-1$.
	Using \cite{FrontalSurfaces} Proposition 5.10 and Conservation of Multiplicity (see e.g. \cite{MondNuno_Singularities}, Corollary E.4), $\mu_{\ms{F}}(\gamma)=\mu_I(\gamma)-\mult(\gamma)+1$, as stated above.
	Therefore,
		\[\mu_{\ms{F}}(\gamma)+\mult(\gamma)-1=\mu_I(\gamma) \geq \codim_{\ms{A}}(\gamma)=\codim_{\ms{F}}(\gamma)+\mult(\gamma)-1,\]
	with equality if $\gamma$ is quasi-homogeneous.
\end{proof}

Now let $f\colon (\mb{K}^n,S) \to (\mb{K}^{n+1},0)$ be a corank $1$ frontal map germ with isolated frontal instability.
We can choose coordinates in the source and target such that
\begin{align*}
	f(x,y)=(x,p(x,y),q(x,y)); && q_y=\mu p_y; && (x,y) \in \mb{K}^{n-1}\times\mb{K},
\end{align*}
for some $p,q, \mu \in \ms{O}_n$.
We then set $S'$ as the projection on the $y$ coordinate of $S$ and consider the \emph{generic slice} $\gamma\colon (\mb{K},S') \to (\mb{K}^2,0)$ of $f$, given by $\gamma(t)=(p(0,t),q(0,t))$.
Since $f$ has isolated frontal instabilities, $\gamma$ is $\ms{A}$-finite (see Proposition \ref{F finite plane curve} above) and we may consider a versal unfolding $\Gamma$ of $\gamma$ with frontal reduction
	\[\Gamma_{\ms{F}}\colon (\mb{K}^d\times \mb{K},S'\times\{0\}) \to (\mb{K}^d\times \mb{K}^2,0).\]

It is not true in general that the sum of two frontal mappings is frontal (e.g. $(x,y)\mapsto (x,y^3,y^4)$ and $(x,y) \mapsto (x,xy,0)$), but we can still construct a \emph{frontal sum} operator that yields a frontal mapping given two frontal mappings with corank at most $1$.
Let $p',q',\mu' \in \ms{O}_{d+1}$ such that
\begin{align*}
	\Gamma_{\ms{F}}(u,y)=(u,p'(u,y),q'(u,y)); 	&& q'_y=\mu p'_y:
\end{align*}
we define the \emph{frontal sum} $F\colon (\mb{K}^d\times\mb{K}^n,\{0\}\times S) \to (\mb{K}^d\times\mb{K}^{n+1},0)$ of $f$ and $\Gamma_{\ms{F}}$ as $F(u,x,y)=(u,x,P(u,x,y),Q(u,x,y))$, where
\begin{equation}\label{frontal sum}
    \begin{aligned}
		P(u,x,y)&=p(x,y)+p'(u,y)-p(0,y);\\
		Q(u,x,y)&=\int^y_0(\mu(x,s)+\mu'(u,s)-\mu(0,s))P_s(u,x,s)\,ds.
    \end{aligned}
\end{equation}
This map germ constitutes an unfolding of both $f$ and $\Gamma_{\ms{F}}$ by construction.
Versality of $\Gamma_{\ms{F}}$ then implies that $F$ is also versal, and thus stable.
Therefore, frontal sums allow us to construct stable frontal unfoldings that are not necessarily versal.

\begin{example}[Frontalised fold surfaces]
	Let $f\colon (\mb{K}^2,0) \to (\mb{K}^3,0)$ be a frontal fold surface given in the form
	\begin{align*}
		f(x,y)=(x,y^2,a_1(x)y^3+a_2(x)y^5+\dots+a_n(x)y^{2n+1}+y^{2n+3});
	\end{align*}
	wherein we assume $a_0,\dots,a_n \in \mb{K}[x]$.
	The function $t \mapsto f(0,t)$ has order $2n+3$, so $f$ can be seen as a smooth $1$-parameter unfolding of the curve
		\[\gamma(t)=(t^2,t^{2n+3}+a_n(0)t^{2n+1}+\dots+a_1(0)t^3).\]
	A frontal miniversal unfolding for $\gamma$ is given by
		\[\Gamma(u,t)=(u,t^2,t^{2n+3}+u_nt^{2n+1}+\dots+u_1t^3),\]
	and we can recover $f$ by setting $u_j(x)=a_j(x)$.
	Taking $(u,x) \mapsto (0,u_1+a_1(x),\dots,u_n+a_n(x))$ gives the stable unfolding
		\[F(u,x,t)=(u,t^2,t^{2n+3}+[u_n+a_n(x)]t^{2n+1}+\dots+[u_1+a_1(x)]t^3).\]
\end{example}

\begin{remark}\label{F(f) is linear}
	The \emph{frontal sum} defined on \eqref{frontal sum} can be used to show that $\ms{F}(f)$ is linear when $f$ has corank at most $1$: first, since $f$ is a corank $1$ frontal, we take coordinates in the source and target such that
	\begin{align*}
		f(x,y)=(x,p(x,y),q(x,y)); && q_y=\mu p_y,
	\end{align*}
	and consider the generic slice $\gamma(t)=(p(0,t),q(0,t))$.

	Let $\xi,\eta \in \ms{F}(f)$ with respective integral $\ms{F}$-curves $F=(f_u,u)$, $G=(g_u,u)$.
	Since $F$ and $G$ are unfoldings of $f$, they may also be regarded as unfoldings of $\gamma$.
	We then consider the frontal sum $H=(u,v,h_{(u,v)})$ of $F$ and $G$, and set $\hat H=(w,\hat h_w)=(w,h_{(w,w)})$.
	Note that the image of $\hat H$ is simply the intersection of the image of $H$ with the hypersurface of equation $u=v$, so $\hat H$ is frontal.
	
	Using the chain rule and Leibniz's integral rule, we see that
		\[\begin{rcases}
			P_w=P_u+P_v\\
			Q_w=Q_u+Q_v
		\end{rcases}
		\implies  \left.\frac{\p \hat h_w}{\p w}\right|_{w=0}=\xi+\eta\]
	and thus $\xi+\eta \in \ms{F}(f)$.
\end{remark}

\section{Stability of frontal map germs}\label{stable unfoldings}
In \S\ref{frontal reductions}, we described a method to generate $\ms{F}$-versal unfoldings of analytic plane curves using pullbacks.
Nonetheless, as pointed out in Remark \ref{reducing stable unfoldings}, the pullback of a stable unfolding is generally not stable as a frontal.

In this section, we describe a technique to generate stable frontal unfoldings, not too dissimilar to the method Mather used to generate all stable map germs.
We also give a classification of all $\ms{F}$-stable proper frontal map germs $(\mb{C}^3,S) \to (\mb{C}^4,0)$ of corank $1$ in \S \ref{classification}, aided by Hefez and Hernandes' Normal Form Theorem for plane curves \cite{HefHer_NFT, HefHer_NFT4}.

Let $f\colon (\mb{C}^n,S) \to (\mb{C}^{n+1},0)$ be a frontal map germ and $\xi \in \ms{F}(f)$.
By definition of $\ms{F}(f)$, $\xi$ is given by a frontal $1$-parameter unfolding $F=(f_t,t)$ of $f$; this is, $F$ verifies that
	\[d(Y\circ f_t)=\sum^n_{i=1}p_i\,d(X_i\circ f_t)+p_0\,dt\]
for some $p_0,\dots,p_n \in \ms{O}_{n+1}$.
If we now consider the vector field germ $\lambda \xi$ with $\lambda \in \ms{O}_n$, $\lambda \xi$ is given by the $1$-parameter unfolding $(\lambda f_t,t)$.
This unfolding is frontal if and only if
\begin{equation}\label{action F(f)}
	d(Y\circ \lambda f_t)=\sum^n_{i=1}q_i\,d(X_i\circ \lambda f_t)+q_0\,dt
\end{equation}
for some $q_0,\dots,q_n \in \ms{O}_{n+1}$.
Expanding on both sides of the equality and rearranging, we see that Equation \eqref{action F(f)} is equivalent to
	\[\lambda\sum^n_{i=1}(q_i-p_i)d(X_i\circ f_t)+(q_0-\lambda p_0)\,dt=[(Y\circ f_t)-\sum^n_{i=1}q_i(X_i\circ f_t)]\,d\lambda.\]
Therefore, the ring $R_f=\{\lambda \in \ms{O}_n: d\lambda \in \ms{O}_nd(f^*\ms{O}_{n+1})\}$ acts on $\ms{F}(f)$ via the usual action.
In particular, $f^*\ms{O}_{n+1} \subseteq R_f$, so $\ms{F}(f)$ is an $\ms{O}_{n+1}$-module via the action $h\xi=(h\circ f)\xi$.

If we assume that $f$ has integral corank $1$ (so that $\ms{F}(f)$ is a $\mb{K}$-vector space), we can define the $\mb{K}$-vector spaces
\begin{align*}
	T\ms{K}_{\ms{F}e}f=tf(\theta_n)+\mf{m}_{n+1}\ms{F}(f); && T^1_{\ms{K}_{\ms{F}e}}f=\frac{\ms{F}(f)}{T\ms{K}_{\ms{F}e}f}.
\end{align*}
We also define the \textbf{frontal $\ms{K}_e$-codimension} $\codim_{\ms{K}_{\ms{F}e}}f$ of $f$ as the dimension of $T^1_{\ms{K}_{\ms{F}e}}f$ in $\mb{K}$, and will say that $f$ is \textbf{$\ms{K}_{\ms{F}e}$-finite} if $\codim_{\ms{K}_{\ms{F}e}}f < \infty$.

\begin{remark}
	The space $\ms{F}(f)$ is not generally a $\ms{O}_n$-module via the usual action: consider the plane curve $\gamma\colon (\mb{K},0) \to (\mb{K}^2,0)$ given by $\gamma(t)=(t^2,t^3)$.
	Using Remark \ref{corank 1 infinitesimal integral deformation}, we see that $(0,1) \in \ms{F}(\gamma)$, but $(0,t)=t(0,1) \not\in \ms{F}(\gamma)$.
\end{remark}

Recall that the Kodaira-Spencer map is defined as the mapping $\overline \omega f\colon T_0\mb{K}^{n+1} \to T^1_{\ms{K}_e}f$ sending $v \in T_0\mb{K}^{n+1}$ onto $\omega f(\eta)$, where $\eta \in \theta_{n+1}$ is such that $\eta_0=v$.
Since $f$ is frontal, the image of $\omega f$ is contained within $\ms{F}(f)$, and the target space becomes $T^1_{\ms{K}_{\ms{F}e}}f$.
Similarly, the kernel of this $\overline \omega f$ becomes
	\[\tau(f):=(\overline\omega f)^{-1}[T\ms{K}_{\ms{F}e}f]|_0,\]
since no element in $T\ms{K}_ef\backslash \ms{F}(f)$ has a preimage.

\begin{lemma}\label{KodairaSpencer}
	The map germ $f$ is $\ms{F}$-stable if and only if $\overline\omega f$ is surjective.
\end{lemma}

\begin{proof}
	Assume $f$ is $\ms{F}$-stable and let $\zeta \in \ms{F}(f)$: there exist $\xi\in \theta_n$ and $\eta \in \theta_{n+1}$ such that $\zeta=tf(\xi)+\omega f(\eta)$.
	Setting $v=\eta_0$, it follows that $\overline\omega f(v) \equiv \zeta \mod T\ms{K}_{\ms{F}e}f$, and surjectivity of $\overline\omega f$ follows.

	Conversely, assume $\overline\omega f$ is surjective: we have the identity
	\begin{equation}
		T\ms{A}_ef+\mf{m}_{n+1}\ms{F}(f)=\ms{F}(f) \label{PreNakayama}
	\end{equation}
	Set $V'=\ms{F}(f)/tf(\theta_{n,S})$ and denote by $p\colon \ms{F}(f) \to V'$ the quotient projection.
	We may then write Equation \eqref{PreNakayama} as
		\[(\pi\circ \omega f)(\theta_{n+1})+\mf{m}_{n+1}V'=V' \implies \frac{V'}{\mf{m}_{n+1} V'} \lesssim (\pi\circ \omega f)(\theta_{n+1}).\]
	Since $(p\circ\omega f)(\theta_{n+1})$ is finitely generated over $\ms{O}_{n+1}$, so is $V'/\mf{m}_{n+1}V'$.
	This implies that $V'/\mf{m}_{n+1}V'$ is finitely generated over $\mb{K}$, so $V'$ is finitely generated over $\ms{O}_{n+1}$ by Weierstrass' Preparation Theorem.
	Since $\ms{O}_{n+1}$ is a local ring, Nakayama's lemma implies that $V'=(\pi\circ \omega f)(\theta_{n+1})$, which is equivalent to $\ms{F}(f)=T\ms{A}_ef$, and frontal stability follows.
\end{proof}

\begin{theorem}\label{stability multigerm}
	A frontal $f\colon (\mb{K}^n,S) \to (\mb{K}^{n+1},0)$ with branches $f_1,\dots,f_r$ is $\ms{F}$-stable if and only if $f_1,\dots,f_r$ are $\ms{F}$-stable and the vector subspaces $\tau(f_1),\dots,\tau(f_r)\subseteq T_0\mb{K}^{n+1}$ meet in general position.
\end{theorem}

\begin{proof}
	Let $g$ be either $f$ or one of its branches.
	By Lemma \ref{KodairaSpencer}, $g$ is $\ms{F}$-stable if and only if $\overline\omega g$ is surjective; this is,
	\begin{equation}\label{stable branches 1}
		\frac{\ms{F}(g)}{T\ms{K}_{\ms{F}_e}g}\cong \frac{T_0\mb{K}^{n+1}}{\ker\overline\omega g}=\frac{T_0\mb{K}^{n+1}}{\tau(g)}
	\end{equation}

	Let $S=\{s_1,\dots,s_r\}$, the ring isomorphism $\ms{O}_{n,S} \to \ms{O}_{n,s_1}\oplus\dots\oplus \ms{O}_{n,s_r}$ induces a module isomorpism $\ms{F}(f) \to \ms{F}(f_1)\oplus\dots \ms{F}(f_r)$, which in turn induces an isomorphism
	\begin{diagram}\label{stable branches 2}
		\frac{\ms{F}(f)}{T\ms{K}_{\ms{F}_e}f} \arrow[r] & \frac{\ms{F}(f_1)}{T\ms{K}_{\ms{F}_e}f_1} \oplus \dots \oplus \frac{\ms{F}(f_r)}{T\ms{K}_{\ms{F}_e}f_r}
	\end{diagram}
	On the other hand, the spaces $\tau(f_i)$ meet in general position if and only if the canonical map
	\begin{diagram}\label{stable branches 3}
		T_0\mb{K}^{n+1} \arrow[r] & \frac{T_0\mb{K}^{n+1}}{\tau(f_1)}\oplus \dots \oplus \frac{T_0\mb{K}^{n+1}}{\tau(f_r)}
	\end{diagram}
	is surjective.
	The statement then follows from (\ref{stable branches 1} - \ref{stable branches 3}).
\end{proof}

We now use Ephraim's theorem to give a geometric interpretation to $\tau(f_i)$, $i=1,\dots,r$.
Recall that the isosingular locus $\Iso(D,x_0)$ of a complex space $D \subseteq W$ at $x_0$ is defined as the germ at $x_0$ of the set of points $x \in D$ such that $(D,x)$ is diffeomorphic to $(D,x_0)$.
Ephraim \cite{Ephraim} showed that $\Iso(D,x_0)$ is a germ of smooth submanifold of $(W,x_0)$ and its tangent space at $x_0$ is given by the evaluation at $x_0$ of the elements in the space
	\[\Der(-\log (D,x_0))=\{\xi \in \theta_W: \xi(I)\subseteq I\}\]
where $I\subset \ms{O}_W$ is the ideal of map germs vanishing on $(D,x_0)$.
We shall now use this result to give a geometric interpretation to the space $\tau(f)$.

\begin{proposition}\label{hat tau is the lift at 0}
	Let $f\colon (\mb{C}^n,S) \to (\mb{C}^{n+1},0)$ be a finite, frontal map germ with integral corank $1$.
	If $f$ is $\ms{F}$-stable and $\codim \Sigma(\tilde f) > 1$, $\tau(f)$ is the tangent space at $0$ of $\Iso(f(\mb{C}^n,S))$.
\end{proposition}

To prove this result, we shall make use of the following

\begin{lemma}[cf. \cite{MondNuno_Singularities}]\label{derlog is the lift}
	Let $f\colon (\mb{C}^n,S) \to (\mb{C}^{n+1},0)$ be a finite, frontal map germ with integral corank $1$ and $\xi \in \theta_{n+1}$.
	If $f$ is $\ms{F}$-finite and $\codim V(p_y,\mu_y) > 1$,
    \[\Der(-\log f)=\Lift(f):=\{\eta \in \theta_{n+1}: \omega f(\eta)=tf(\xi) \text{ for some }\xi \in \theta_n\}.\]
\end{lemma}

\begin{proof}[Proof of Proposition \ref{hat tau is the lift at 0}]
	By Ephraim's theorem \cite{Ephraim}, the tangent space to $\Iso(f(\mb{C}^n,S))$ at $0$ is given by the evaluation at $0$ of the elements in $\Der(-\log f)$.
	Using Lemma \ref{derlog is the lift}, $\Der(-\log f)$ is the space of elements in $\theta_{n+1}$ that are liftable via $f$.
	Therefore, we only need to show that the evaluation of $0$ of this space coincides with $\tau(f)$.

	Let $\eta \in \Lift(f)$: there exists a $\xi \in \theta_n$ such that $\omega f(\eta)=tf(\xi) \in T\ms{K}_{\ms{F}_e}f$, so $\eta|_0 \in \tau(f)$.
	Conversely, if $\eta \in \theta_{n+1}$ verifies that $\eta|_0 \in \tau(f)$, there exist $\xi\in \theta_n$, $\zeta \in \ms{F}(f)$ such that
		\[\omega f(\eta)=tf(\xi)+(f^*\beta)\zeta\]
	for some $\beta \in \mf{m}_{n+1}$.
	Since $f$ is $\ms{F}$-stable, $\ms{F}(f)=T\ms{A}_ef$, which implies that
		\[(f^*\mf{m}_{n+1})\ms{F}(f)=
			(f^*\mf{m}_{n+1})[tf(\theta_n)+\omega f(\theta_{n+1})]\subseteq
			tf(\mf{m}_n\theta_n)+\omega f(\mf{m}_{n+1}\theta_{n+1})\]
	Therefore, there exist $\xi' \in \mf{m}_n\theta_n$ and $\eta' \in \mf{m}_{n+1}\theta_{n+1}$ such that
		\[(f^*\beta)\zeta=tf(\xi')+\omega f(\eta') \implies \omega f(\eta-\eta')=tf(\xi+\xi')\]
	and $\eta-\eta' \in \Lift(f)$.
	In particular, if $s \in S$, $(\eta-\eta')|_0=\omega f(\eta-\eta')|_s=v-0=v$, thus finishing the proof.
\end{proof}

\subsection{Generating stable frontal unfoldings}
The generation of stable unfoldings in Thom-Mather's theory of smooth deformations is done by computing the $\ms{K}_e$-tangent space of a smooth map germ $f\colon (\mb{K}^n,0) \to (\mb{K}^p,0)$ of rank $0$.
If $\mf{m}_n\theta(f)/T\ms{K}_ef$ is generated over $\mb{K}$ by the classes of $g_1,\dots,g_s\in \ms{O}_n$, Martinet's theorem (\cite{MondNuno_Singularities}, Theorem 7.2) states that the map germ
	\[F(u,x)=(u,f(x)+u_1g_1(x)+\dots+u_sg_s(x))\]
is a stable unfolding of $f$.
While such a result fails to yield frontal unfoldings of frontal map germs, if $f$ has corank $1$, we can still make use of the frontal sum operation defined on \S \ref{frontal reductions} to formulate a frontal version of Martinet's theorem.

\begin{lemma}\label{stable unfolding lemma}
	Let $f\colon (\mb{K}^n,0) \to (\mb{K}^{n+1},0)$ be a frontal map germ of integral corank $1$ with frontal unfolding $F=(u,f_u)$, and $(u,y)$ be local coordinates on $(\mb{K}^d\times \mb{K}^{n+1},0)$.
	There is an $\ms{O}_{n+d+1}$-linear isomorphism
		\[\beta\colon \frac{\ms{F}(F)}{T\ms{K}_{\ms{F}e}F} \longrightarrow \frac{\ms{F}(f)}{T\ms{K}_{\ms{F}e}f}\]
	induced by the $\ms{O}_{n+d}$-linear epimorphism $\beta_0\colon \theta(F) \to \theta(f)$ sending $\p y_i$ onto $\p y_i$ for $i=1,\dots,n+1$ and $\p u_j$ onto $-\dot F_j$ for $j=1,\dots,d$.
\end{lemma}

\begin{proof}
	In \cite{MondNuno_Singularities}, Lemma 5.5, it is shown that $\beta_0$ induces a $\ms{O}_{n+d}$-linear isomorphism $\beta_1\colon T^1_{\ms{K}_e}F \to T^1_{\ms{K}_e}f$.
	In particular, we can consider $\beta_0$ as a $\ms{O}_{n+d+1}$-epimorphism via $F^*$.
	Note that $T\ms{K}_{\ms{F}e}g=T\ms{K}_eg\cap \ms{F}(g)$ for any frontal map germ $g$ with integral corank $1$, so it suffices to show that $\beta_0$ sends $\ms{F}(F)$ onto $\ms{F}(f)$.

	Let $\xi \in \theta(F)$ with integral $\ms{F}$-curve $F_t$: the integral $\ms{F}$-curve for $\beta_0(\xi)$ is given by
	\begin{align*}
		f_t=i^*(\pi\circ F_t); && \pi(t,u,y)=(t,y); && i(x)=(0,x).
	\end{align*}
	In particular, if $(t,F_t)$ is a frontal, $(t,f_t)$ is also frontal, since the image of $(t,f_t)$ is embedded within the image of $(t,F_t)$.
	Conversely, given a frontal unfolding $(t,f_t)$ of $f$, the map $(t,u,f_t)$ is a frontal unfolding of $F$ with $f_t=i^*(\pi\circ F_t)$, hence $\beta_0(\ms{F}(F))=\ms{F}(f)$.
\end{proof}

As a consequence of Lemma \ref{stable unfolding lemma}, if $f\colon (\mb{K}^n,0) \to (\mb{K}^{n+1},0)$ is a stable frontal map germ, it is either the versal unfolding of some frontal map germ of rank $0$ or a prism (i.e. a trivial unfolding) thereof.

\begin{theorem}\label{bounds for KF}
	Let $\gamma\colon (\mb{K},0) \to (\mb{K}^2,0)$ be the plane curve from Remark \ref{Piuseux theorem}, and
	\begin{align*}
		T_j(t)=(t^j,B_j(t)), && B_j(t)=j\int^t_0s^{j-1}\mu(s)\,ds.
	\end{align*}
	If $\ms{F}_0(\gamma)=\ms{F}(\gamma)\cap \mf{m}_1\theta(\gamma)$, then
		\[\Sp_{\mb{K}}\{T_1,\dots,T_{\alpha-2}\} \hookrightarrow \frac{\ms{F}_0(\gamma)}{T\ms{K}_{\ms{F}e}\gamma} \hookrightarrow \Sp_{\mb{K}}\{T_1,\dots,T_{\alpha-2},(0,t^\alpha),\dots,(0,t^{2\alpha-1})\}.\]
\end{theorem}

\begin{proof}
	Let $\xi=(a,b) \in \theta(\gamma)$: by Remark \ref{corank 1 infinitesimal integral deformation}, $\xi \in \ms{F}(\gamma)$ if and only if $b'-\mu a' \in \mf{m}_1^{\alpha-1}$, which in turn is equivalent to assuming that $b'-\mu a' \equiv \lambda_1 T'_1+\dots+\lambda_{\alpha-2}T'_{\alpha-2} \mod \mf{m}_1^{\alpha-1}$ for some $\lambda_1,\dots,\lambda_{\alpha-2}\in \mb{K}$.
	Therefore,
	\begin{equation}\label{F gamma decomposition}
		\ms{F}(\gamma)=\mb{K}\oplus\Sp_{\mb{K}}\left\{T_1,\dots,T_{\alpha-1}\right\}\oplus\mf{m}_1^\alpha\theta(\gamma).
	\end{equation}

	A simple computation shows that $T\ms{K}_{\ms{F}e}\gamma \subseteq \mf{m}^{\alpha-1}_1\theta(\gamma)$, hence $T_j \not\in T\ms{K}_{\ms{F}e}\gamma$ for $j < \alpha-1$.
	However, $T_{\alpha-1} \in t\gamma(\theta_1)$, giving the first monomorphism.
	For the second monomorphism, first note that $\gamma$ is finitely determined, so there exists a $k > 0$ such that $\mf{m}_1^{k+1}\theta(\gamma) \subseteq T\ms{A}_e\gamma \subseteq T\ms{K}_{\ms{F}e}\gamma$.
	If $j=\alpha,\dots,k$, there exist $l > 0$ and $0 \leq \beta < \alpha$ such that $j=l\alpha+\beta$.
	Using Equation \eqref{F gamma decomposition}, we see that
		\[(t^j,0)=(t^\alpha)^l(t^\beta,0)=(t^\alpha)^lT_\beta(t)+(t^\alpha)^l(0,B_\beta(t)) \in \mf{m}_2\ms{F}(\gamma) \subseteq T\ms{K}_{\ms{F}e}\gamma.\]
	Similarly, $(0,t^j) \in T\ms{K}_{\ms{F}e}\gamma$ for all $j \geq 2\alpha$.
\end{proof}

If we now consider the $1$-parameter unfolding $\Gamma_j(u,t)=(u,\gamma(t)+uT_j(t))$,
	\[\cvf{t}(t^{\alpha+1}h(t)+uB_j(t))=\mu(t)\cvf{t}(t^\alpha+jut^j)\]
and $\Gamma_j$ is frontal due to Corollary \ref{pq criterion frontality}.
Similarly, if we set $\Gamma_k(u,t)=(u,\gamma(t)+ut^\alpha k(t))$ with $k \in \ms{O}_1^2$,
\begin{align*}
	Q_t(u,t)=\cvf{t}(t^{\alpha+1}h(t)+ut^\alpha k_2(t))	&=t^{\alpha-1}(\alpha\mu(t)+uk_2(t)+utk'_2(t));\\
	P_t(u,t)=\cvf{t}(t^{\alpha}+ut^\alpha k_1(t))		&=t^{\alpha-1}(\alpha+\alpha uk_1(t)+tk_1'(t)).
\end{align*}
Since $\alpha+\alpha uk_1(t)+tk_1'(t)$ is a unit, $P_t\,|\, Q_t$ and $\Gamma_k$ is also frontal.

If $\ms{F}_0(\gamma)=T\ms{K}_{\ms{F}e}\gamma+\Sp_{\mb{K}}\{T_{j_1},\dots,T_{j_d},k_1,\dots,k_b\}$ for some $k_1,\dots,k_b \in \mf{m}^\alpha_1\ms{O}_1^2$, we consider the $(d+b)$-parameter frontal unfolding
\begin{equation}\label{frontal sum KF}
	F(u,t)=\Gamma_{j_1}(u_1,t)\#\dots\#\Gamma_{j_d}(u_d,t)\#\Gamma_{k_1}(u_{d+1},t)\#\dots\#\Gamma_{k_b}(u_{d+b},t),
\end{equation}
where $\#$ denotes the frontal sum operation defined on Equation \eqref{frontal sum}.

\begin{example}\label{generating a stable unfolding of the E8}
	Let $f\colon (\mb{K},0) \to (\mb{K}^2,0)$ be the plane curve $f(t)=(t^3,t^5)$, which verifies that $\ms{F}_0(f)=T\ms{K}_{\ms{F}e}f\oplus \Sp_{\mb{K}}\{(9t,5t^3),(0,t^4)\}$.
	We then consider the $1$-parameter unfoldings
	\begin{align*}
		F_1(t,v)=(v,t^3,t^5+vt^4); && F_2(t,u)=(u,t^3+9ut,t^5+5ut^3),
	\end{align*}
	whose frontal sum is
		\[F(t,u,v)=\left(u,v,t^3+9ut,t^5+5ut^3+\frac{1}{3}vt^4+6uvt^2\right).\]
	This unfolding is $\ms{A}$-equivalent to the $A_{3,1}$ singularity from \cite{Ishikawa05}, Example 4.2.
\end{example}

\begin{theorem}\label{existence of minimal stable unfolding}
	The map germ $F\colon (\mb{K}^d\times\mb{K}^b\times \mb{K},0) \to (\mb{K}^d\times\mb{K}^b\times \mb{K}^2,0)$ defined on Equation \eqref{frontal sum KF} is stable as a frontal.
	Moreover, if the $\mb{K}$-codimension of $T\ms{K}_{\ms{F}_e}f$ over $\ms{F}_0(f)$ is $d+b$, every other stable frontal unfolding of $f$ must have at least $d+b$ parameters.
\end{theorem}

\begin{proof}
	It is clear by definition of $T\ms{K}_{\ms{F}e}F$ that
		\[\ms{F}_0(F)\supseteq T\ms{K}_{\ms{F}e}F\supseteq T\ms{A}_eF\cap \mf{m}_{d+b+1}\theta(F),\]
	so $F$ is $\ms{F}$-stable if and only if $\ms{F}_0(F)=T\ms{K}_{\ms{F}e}F$.
	By Lemma \ref{stable unfolding lemma}, this is equivalent to
	\begin{equation*}
		\ms{F}_0(f)=T\ms{K}_{\ms{F}e}f+\Sp_{\mb{K}}\left\{-\dot F_1, \dots,-\dot F_{d+b}\right\}.
	\end{equation*}
	It follows from the definition of frontal sum that
		\[\dot F_i(t)=(P_{u_i}(0,t),Q_{u_i}(0,t))=
		\begin{cases}
			\dot \Gamma_{j_i}(t) 	& \text{ if } i\leq d;\\
			\dot \Gamma_{k_{i}}(t) 	& \text{ if } i > d,
		\end{cases}\]
	and thus $F$ is stable.
\end{proof}

\subsection{Corank $1$ stable frontal map germs in dimension $3$}\label{classification}
By Theorem \ref{stability multigerm}, a frontal multigerm $f\colon (\mb{K}^3,S) \to (\mb{K}^4,0)$ is $\ms{F}$-stable if and only if its branches $f_1,\dots,f_r$ are $\ms{F}$-stable and $\tau(f_1),\dots,\tau(f_r)$ meet in general position.
Therefore, we only need to classify the stable monogerms.

By Lemma \ref{stable unfolding lemma}, every $\ms{F}$-stable monogerm with corank $1$ is a versal unfolding of an irreducible analytic plane curve $\gamma$ with $\ms{F}_e$-codimension at most $2$.
In particular, if $\gamma(\mb{C},0)$ is the zero locus of some analytic $g \in \ms{O}_2$, $\tau(g)-\delta(g)\leq \ord(g)+1$ due to Corollary \ref{formula for frontal codimension}.
A consequence of Theorem \ref{bounds for KF} is that $\codim_{\ms{K}_{\ms{F}e}}\gamma \geq \ord(g)$, meaning that $\ord(g)$ must be at most $4$.

If $\ord(g)=2$, it follows from a result by Zariski \cite{Zariski} that $g(x,y)=x^2-y^{2n+1}$.
For $n=0,1$, this yields an $\ms{F}$-stable plane curve; for $n > 1$, we can unfold $\gamma(t)$ into
	\[\Gamma_n(u,t)=(u,t^2,t^{2n+1}+ut^3),\]
which is stable.

The cases $\ord(g)=3$ and $\ord(g)=4$ will be examined using Hefez and Hernandes' classification of analytic plane curves from \cite{HefHer_NFT4}.
Every analytic plane curve has an associated invariant $\Sigma=\la v_0,\dots, v_g\ra$, known as the \textbf{semigroup of values}.
If the curve is irreducible, its delta invariant $\delta$ is equal to
	\[\frac{1}{2}\left[1-v_0-\sum^g_{i=1}v_i\left(1-\frac{\GCD(v_0,\dots,v_{i-1})}{\GCD(v_0,\dots,v_i)}\right)\right],\]
regardless of its analytic family.
Therefore, the expression $\tau-\delta$ only depends on $\tau$.

For $\ord(g)=3$, $\Sigma$ is given by $\la 3, v_1\ra$ with $v_1 > 3$, so $\delta=v_1-1$.
If $\tau=2(v_1-1)$, $\tau-\delta=v_1-1 < 4$, so $g(x,y)$ is either $x^3-y^4$ or $x^3-y^5$.
The case $\tau=2v_1-j-1$ with $j \geq 2$ implies that $\tau < \delta$, which is impossible.

For $\ord(g)=4$, $\Sigma$ can be either $\la 4,v_1\ra$ or $\la 4,v_1,v_2\ra$.
If $\Sigma=\la 4,v_1\ra$, $v_1$ is coprime with $4$, so $\delta=3/2(v_1-1)$ and we have two possible values for $\tau$:
\begin{enumerate}
	\item if $\tau=3(v_1-1)$, $\tau-\delta=3/2(v_1-1) \leq 5$, which implies that $\tau < \delta$;
	\item if $\tau=3v_1-j-2$ with $j > 1$,
		\[\tau-\delta=\frac{1}{2}(3v_1-2j-1) \leq 5 \implies j \geq \frac{1}{2}(3v_1-11).\]
	Since $j \leq v_1/2$, it follows that $v_1 \geq 3v_1-11$, giving us $\gamma(t)=(t^4,t^5+t^7)$.
\end{enumerate}
If $\Sigma=\la 4,v_1,v_2\ra$, $\GCD(4,v_1)=2$ and $\GCD(4,v_1,v_2)=1$, which implies that $v_1 \geq 6$ and $v_2 \geq 2v_1$.
Using
\begin{align*}
	\delta=\frac{1}{2}(v_2+v_1-3); && \tau=v_2+\frac{1}{2}v_1-2,
\end{align*}
it follows that $\tau-\delta=(v_2-1)/2 > 5$.
Since we are only interested in the case $\tau-\delta \leq 5$, we can ignore this case.

For the remaining cases, the possible values for $\tau-\delta$ fall into one the following categories:
\begin{align*}
	\frac{3(v_1-1)}{2}+k-\left[\frac{v_1}{4}\right]; && \frac{3(v_1-1)}{2}-2j+1; && \frac{3(v_1-1)}{2}-2j+2,
\end{align*}
for $2 \leq j \leq [v_1/4]$ and $1 \leq k \leq [v_1/4]-j$.
If $\tau-\delta \leq 5$, then $v_1 \geq 7$, which is not possible.

\begin{table}
\centering
\begin{tabular}{LCL}
	\text{Plane curve}	& \multicolumn{2}{c}{Versal frontal unfolding} 						\\
	\hline\hline
	(t^2,t^3) 		\M	& A_{2,0}	& (u,v,t^2,t^3)											\\
	(t^2,t^5)       \M  & A_{2,1}	& (u,v,t^2,t^5+ut^3)									\\
	(t^3,t^4)       \M  & A_{3,0}	& (u,v,t^3+3ut, 3t^4+2ut^2)								\\
	(t^3,t^5) 	    \M  & A_{3,1} 	& (u,v,t^3+tu, t^5+vt^4+2uvt^2-5u^2t)					\\
	(t^4,t^5+t^7) 	\M  & A_{4,0} 	& (u,v,t^4+8tu,t^7+t^5+t^3v(5-14v)+t^2u(5-42v)-28tu^2)	\\
\end{tabular}
	\caption{\label{tabla estables} Stable proper frontal map germs $(\mb{C}^3,0) \to (\mb{C}^4,0)$.
	The notation $A_{i,j}$ is due to Ishikawa \cite{Ishikawa05}.}
\end{table}

\begin{theorem}
	Table \ref{tabla estables} shows all stable proper frontal map germs $(\mb{C}^3,0) \to (\mb{C}^4,0)$ of corank $1$ together with the plane curves of which they are versal unfoldings.
	All stable frontal multigerms are obtained by transverse self-intersections of these mono-germs, as shown in Theorem \ref{stability multigerm}.
\end{theorem}

\begin{proof}
	The discussion conducted throughout this subsection shows that the only plane curves of frontal codimension less than or equal to 2 are $(t^2,t^3)$, $(t^2,t^5)$, $(t^2,t^7)$, $(t^3,t^4)$, $(t^3,t^5)$ and $(t^4,t^5+t^7)$.
	
	The curve $(t^2,t^3)$ is easily checked to be stable as a frontal.
	The family of curves $(t^2,t^{2k+1})$ for $k > 1$ unfolds into $(s,t) \mapsto (s,t^2,t^{2k+1}+st^3)$, which is $\ms{A}$-equivalent to the folded Whitney umbrella $(s,t) \mapsto (s,t^2,st^3)$,which is stable as a frontal (\cite{FrontalSurfaces,Nuno_CuspsAndNodes}).

	The curves $(t^3,t^4)$ and $(t^4,t^5+t^7)$ unfold into the swallowtail and butterfly singularities ($A_{3,0}$ and $A_{4,0}$ in Table \ref{tabla estables}), both of which are stable wave fronts (\cite{Arnold_I}).
	The $E_8$ singularity unfolds into Ishikawa's $A_{3,1}$ singularity \cite{Ishikawa05}.
\end{proof}

\begin{conjecture}
	The stable proper frontal map germs $f\colon (\mb{C}^n,0) \to (\mb{C}^{n+1},0)$ of corank $1$ are given by Ishikawa's $A_{i,j}$ singularities, where
		\begin{align*}
			i=\dim \frac{\tilde f^*\ms{O}_{2n+1}}{f^*\mf{m}_{n+1}} \in \{2,\dots,n\}; && j+1=\dim \frac{\ms{O}_{n}}{\tilde f^*\mf{m}_{2n+1}} \in \left\{1,\dots,\left[\frac{n}{2}\right]\right\},
		\end{align*}
	where square brackets denote the floor function.
	All stable frontal multigerms are obtained by transverse self-intersections of these mono-germs, as shown in Theorem \ref{stability multigerm}.
\end{conjecture}

The algebra $\tilde f^*\ms{O}_{2n+1}/f^*\mf{m}_{n+1}$ was introduced by Ishikawa in \cite{Ishikawa05} in order to give a characterisation of Legendrian stability

\section{Acknowledgements}
We would like to thank M. E. Hernandes for his helpful contributions to \S \ref{classification}.


\begin{thebibliography}{10}

\bibitem{Arnold_CWF}
V.~I. Arnol'd.
\newblock {\em Singularities of caustics and wave fronts}, volume~62 of {\em
  Mathematics and its Applications (Soviet Series)}.
\newblock Kluwer Academic Publishers Group, Dordrecht, 1990.

\bibitem{BurVerona72}
Dan Burghelea and Andrei Verona.
\newblock Local homological properties of analytic sets.
\newblock {\em Manuscripta Math.}, 7:55--66, 1972.

\bibitem{Casals_Murphy}
Roger Casals and Emmy Murphy.
\newblock Legendrian fronts for affine varieties.
\newblock {\em Duke Math. J.}, 168(2):225--323, 2019.

\bibitem{ChenPeiTakahashi}
L.~Chen, D.~H. Pei, and M.~Takahashi.
\newblock Dualities and envelopes of one-parameter families of frontals in
  hyperbolic and de {S}itter 2-spaces.
\newblock {\em Math. Nachr.}, 293(5):893--909, 2020.

\bibitem{deJongPfister}
Theo de~Jong and Gerhard Pfister.
\newblock {\em Local analytic geometry}.
\newblock Advanced Lectures in Mathematics. Friedr. Vieweg \& Sohn,
  Braunschweig, 2000.
\newblock Basic theory and applications.

\bibitem{Ephraim}
Robert Ephraim.
\newblock Isosingular loci and the {C}artesian product structure of complex
  analytic singularities.
\newblock {\em Trans. Amer. Math. Soc.}, 241:357--371, 1978.

\bibitem{Shoichi}
Shoichi Fujimori, Kentaro Saji, Masaaki Umehara, and Kotaro Yamada.
\newblock Singularities of maximal surfaces.
\newblock {\em Math. Z.}, 259(4):827--848, 2008.

\bibitem{GreuelLossenShustin}
G.-M. Greuel, C.~Lossen, and E.~Shustin.
\newblock {\em Introduction to singularities and deformations}.
\newblock Springer Monographs in Mathematics. Springer, Berlin, 2007.

\bibitem{HefHer_NFT}
A.~Hefez and M.~E. Hernandes.
\newblock The analytic classification of plane branches.
\newblock {\em Bull. Lond. Math. Soc.}, 43(2):289--298, 2011.

\bibitem{HefHer_NFT4}
Abramo Hefez and Marcelo~Escudeiro Hernandes.
\newblock Analytic classification of plane branches up to multiplicity 4.
\newblock {\em J. Symbolic Comput.}, 44(6):626--634, 2009.

\bibitem{Ishikawa96}
Go-o Ishikawa.
\newblock Symplectic and lagrange stabilities of open whitney umbrellas.
\newblock {\em Inventiones mathematicae}, 126(2):215--234, 1996.

\bibitem{Ishikawa83}
Goo Ishikawa.
\newblock Families of functions dominated by distributions of $c$-classes of
  mappings.
\newblock In {\em Annales de l'institut Fourier}, volume~33, pages 199--217,
  1983.

\bibitem{Ishikawa94}
Goo Ishikawa.
\newblock Parametrized legendre and lagrange varieties.
\newblock {\em Kodai Mathematical Journal}, 17(3):442--451, 1994.

\bibitem{Ishikawa05}
Goo Ishikawa.
\newblock Infinitesimal deformations and stability of singular legendre
  submanifolds.
\newblock {\em Asian Journal of Mathematics}, 9, 03 2005.

\bibitem{Ishikawa_Survey}
Goo Ishikawa.
\newblock Recognition problem of frontal singularities.
\newblock {\em J. Singul.}, 21:149--166, 2020.

\bibitem{MartinsNuno}
R.~Martins and J.~J. Nu\~{n}o Ballesteros.
\newblock The link of a frontal surface singularity.
\newblock In {\em Real and complex singularities}, volume 675 of {\em Contemp.
  Math.}, pages 181--195. Amer. Math. Soc., Providence, RI, 2016.

\bibitem{Milnor_RedBook}
John Milnor.
\newblock {\em Singular points of complex hypersurfaces}.
\newblock Annals of Mathematics Studies, No. 61. Princeton University Press,
  Princeton, N.J.; University of Tokyo Press, Tokyo, 1968.

\bibitem{Mond_Classification}
David Mond.
\newblock On the classification of germs of maps from {${\bf R}^2$} to {${\bf
  R}^3$}.
\newblock {\em Proc. London Math. Soc. (3)}, 50(2):333--369, 1985.

\bibitem{Mond_Vanish}
David Mond.
\newblock Vanishing cycles for analytic maps.
\newblock In {\em Singularity theory and its applications, {P}art {I}
  ({C}oventry, 1988/1989)}, volume 1462 of {\em Lecture Notes in Math.}, pages
  221--234. Springer, Berlin, 1991.

\bibitem{MondNuno_Singularities}
David Mond and Juan~J. Nu\~{n}o Ballesteros.
\newblock {\em Singularities of mappings---the local behaviour of smooth and
  complex analytic mappings}, volume 357 of {\em Grundlehren der mathematischen
  Wissenschaften [Fundamental Principles of Mathematical Sciences]}.
\newblock Springer, Cham, [2020] \copyright 2020.

\bibitem{MurataUmehara}
Satoko Murata and Masaaki Umehara.
\newblock Flat surfaces with singularities in {E}uclidean 3-space.
\newblock {\em J. Differential Geom.}, 82(2):279--316, 2009.

\bibitem{FrontalSurfaces}
C.~Muñoz-Cabello, J.~J. Nuño-Ballesteros, and R.~Oset Sinha.
\newblock Singularities of frontal surfaces, 2022.
\newblock arXiv:2205.02097.
\newblock To appear in Hokkaido Math. J,.

\bibitem{Nuno_CuspsAndNodes}
Juan~J. {Nu{\~n}o Ballesteros}.
\newblock Unfolding plane curves with cusps and nodes.
\newblock {\em Proc. Roy. Soc. Edinburgh Sect. A}, 145(1):161--174, 2015.

\bibitem{OsetSaji}
Ra\'{u}l Oset~Sinha and Kentaro Saji.
\newblock On the geometry of folded cuspidal edges.
\newblock {\em Rev. Mat. Complut.}, 31(3):627--650, 2018.

\bibitem{SUYGaussBonnet}
Kentaro Saji, Masaaki Umehara, and Kotaro Yamada.
\newblock The geometry of fronts.
\newblock {\em Ann. of Math. (2)}, 169(2):491--529, 2009.

\bibitem{Spanier}
Edwin~H. Spanier.
\newblock {\em Algebraic topology}.
\newblock Springer-Verlag, New York-Berlin, 1981.
\newblock Corrected reprint.

\bibitem{Arnold_I}
A.N. Varchenko~(auth.) {V.I. Arnold}, S.M. Gusein-Zade.
\newblock {\em Singularities of Differentiable Maps, Volume 1: Classification
  of Critical Points, Caustics and Wave Fronts}.
\newblock Modern Birkh{\"a}user Classics. Birkh{\"a}user Basel, 1 edition,
  2012.

\bibitem{Wall}
C.~T.~C. Wall.
\newblock {\em Singular points of plane curves}, volume~63 of {\em London
  Mathematical Society Student Texts}.
\newblock Cambridge University Press, Cambridge, 2004.

\bibitem{ZakalyukinKurbatskii}
V.~M. Zakalyukin and A.~N. Kurbatski\u{\i}.
\newblock Envelope singularities of families of planes in control theory.
\newblock {\em Tr. Mat. Inst. Steklova}, 262(Optim. Upr.):73--86, 2008.

\bibitem{Zariski}
Oscar Zariski.
\newblock Characterization of plane algebroid curves whose module of
  differentials has maximum torsion.
\newblock {\em Proc. Nat. Acad. Sci. U.S.A.}, 56:781--786, 1966.

\end{thebibliography}
\end{document}